\documentclass[a4paper,11pt,reqno]{amsart}
\pdfoutput=1

\usepackage{amsmath,amssymb,amsthm,mathrsfs}
\usepackage{mathtools}
\usepackage[all,cmtip]{xy}
\usepackage{bm}
\usepackage{afterpage}
\usepackage[a4paper,margin=1in]{geometry}
\usepackage{lmodern}
\usepackage[T1]{fontenc}
\usepackage{microtype}

\usepackage[UKenglish]{isodate}

\makeatletter
\@namedef{subjclassname@2020}{\textup{2020} Mathematics Subject Classification}
\makeatother

\usepackage[shortlabels]{enumitem}

\theoremstyle{plain}

\newtheorem{theorem}{Theorem}[section]
\newtheorem{proposition}[theorem]{Proposition}
\newtheorem{lemma}[theorem]{Lemma}
\newtheorem{corollary}[theorem]{Corollary}

\theoremstyle{definition} 
\newtheorem{definition}[theorem]{Definition} 
\newtheorem{example}[theorem]{Example}

\theoremstyle{remark}
\newtheorem{remark}[theorem]{Remark}

\makeatletter
\let\c@equation\c@theorem
\makeatother
\numberwithin{equation}{section}

\usepackage[hidelinks]{hyperref}

\newcommand{\lra}{\longrightarrow}
\newcommand{\Cat}{\mathbf{Cat}}

\newcommand{\Set}{\mathbf{Set}}
\newcommand{\sSet}{\mathbf{sSet}}
\newcommand{\Hom}{\mathrm{Hom}}

\newcommand{\Ho}{\operatorname{Ho}}
\newcommand{\ho}{\operatorname{ho}}
\newcommand{\Bicat}{\operatorname{\mathbf{Bicat}}}

\newcommand{\sk}{\mathrm{sk}}
\newcommand{\cosk}{\mathrm{cosk}}

\newcommand{\cd}[2][]{\vcenter{\hbox{\xymatrix#1{#2}}}}

\newcommand{\dtwocell}[3][0.5]{\ar@{}[#2] \ar@{=>}?(#1)+/u  0.2cm/;?(#1)+/d 0.2cm/^{#3}}
\newcommand{\ltwocell}[3][0.5]{\ar@{}[#2] \ar@{=>}?(#1)+/r 0.2cm/;?(#1)+/l 0.2cm/_{#3}}

\newcommand{\pushoutcorner}[1][dr]{\save*!/#1+1.2pc/#1:(1,-1)@^{|-}\restore}

\newcommand{\fatpullbackcorner}[1][dr]{\save*!/#1-1.75pc/#1:(-1,1)@^{|-}\restore}
\newcommand{\fatpushoutcorner}[1][dr]{\save*!/#1+1.75pc/#1:(1,-1)@^{|-}\restore}

\makeatletter
\def\matrixobject@{%
  \edef \next@{={\DirectionfromtheDirection@ }}%
  \expandafter \toks@ \next@ \plainxy@
  \let\xy@@ix@=\xyq@@toksix@
  \xyFN@ \OBJECT@}
\let\xy@entry@@norm=\entry@@norm
\def\entry@@norm@patched{%
  \let\object@=\matrixobject@
  \xy@entry@@norm }
\AtBeginDocument{\let\entry@@norm\entry@@norm@patched}
\makeatother

\newcommand{\hdash}{\rotatebox[origin=c]{90}{$\vdash$}}

\title{On truncated quasi-categories}
\author{Alexander Campbell}
\address{Centre of Australian Category Theory \\ Macquarie University \\ NSW 2109 \\ Australia}
\urladdr{http://web.science.mq.edu.au/~alexc/ \\ https://sites.google.com/view/edoardo-lanari/}
\author{Edoardo Lanari}

\date{23 January 2020}

\begin{document}

\begin{abstract} For each $n \geq -1$, a quasi-category is said to be $n$-truncated if its hom-spaces are $(n-1)$-types. In this paper we study the model structure for $n$-truncated quasi-categories, which we prove can be constructed as the Bousfield localisation of Joyal's model structure for quasi-categories with respect to the boundary inclusion of the $(n+2)$-simplex. Furthermore, we prove the expected Quillen equivalences between categories and $1$-truncated quasi-categories and between $n$-truncated quasi-categories and Rezk's $(n,1)$-$\Theta$-spaces.
\end{abstract}

\maketitle

\tableofcontents

\section{Introduction} \label{secintro}
Quasi-categories were introduced by Boardman and Vogt \cite[\S IV.2]{MR0420609}, and were developed by Joyal \cite{MR1935979,joyalbarcelona} and Lurie \cite{MR2522659} among others as a model for $(\infty,1)$-categories:\  (weak) infinite-dimensional categories in which every morphism above dimension $1$ is (weakly) invertible. Among the $(\infty,1)$-categories are the $(n,1)$-categories, which have no non-identity morphisms above dimension $n$.\footnote{This description is accurate only for $n \geq 1$; it is natural to identify $(0,1)$-categories with posets and $(-1,1)$-categories with truth values (i.e.\ $0$ and $1$). See \cite{MR2664619} for a discussion of this point.} In \cite[\S2.3.4]{MR2522659}, Lurie identified the quasi-categories that model $(n,1)$-categories (for $n\geq 1$) as those in which every inner horn above dimension $n$ has a unique filler. Moreover, he proved that a quasi-category is equivalent to such a quasi-category precisely when its hom-spaces are homotopy $(n-1)$-types (i.e.\ Kan complexes whose homotopy groups are trivial above dimension $n-1$);  in \cite[\S26]{joyalnotes}, Joyal called quasi-categories with this latter property  \emph{$n$-truncated}, and stated without proof a collection of assertions on $n$-truncated quasi-categories.
 
In this paper, we prove  (Theorem \ref{newthm}) that, for each $n \geq -1$, the $n$-truncated quasi-categories are the fibrant objects of the Bousfield localisation of Joyal's model structure for quasi-categories with respect to the boundary inclusion $\partial\Delta^{n+2} \lra \Delta^{n+2}$. 
(Note that the existence of the model structure for $n$-truncated quasi-categories was stated without proof in Joyal's notes \cite[\S26.5]{joyalnotes}. However, our construction and identification of this model structure as the Bousfield localisation of Joyal's model structure for quasi-categories with respect to the boundary inclusion $\partial\Delta^{n+2} \lra \Delta^{n+2}$ is new to this paper; see Remark \ref{nosmall}.) Moreover, we prove  (Theorem \ref{wethm}) Joyal's assertion (stated without proof in \cite[\S26.6]{joyalnotes}) that, if $n \geq 1$, a morphism of quasi-categories is a weak equivalence in this model structure if and only if it is essentially surjective on objects and an $(n-1)$-equivalence on hom-spaces.

Furthermore, we prove (Theorem \ref{jtrezkthm})  that the two Quillen equivalences 
\begin{equation*}
\xymatrix{
[\Delta^\mathrm{op},\Set] \ar@<-1.5ex>[rr]^-{\hdash}_-{t^!} && \ar@<-1.5ex>[ll]_-{t_!} [(\Delta\times\Delta)^\mathrm{op},\Set]
}
\qquad\quad
\xymatrix{
[(\Delta\times\Delta)^\mathrm{op},\Set] \ar@<-1.5ex>[rr]^-{\hdash}_-{i_1^*} && \ar@<-1.5ex>[ll]_-{p_1^*} [\Delta^\mathrm{op},\Set]
}
\end{equation*}
established by Joyal and Tierney \cite{MR2342834} between the model structures for quasi-categories and complete Segal spaces remain Quillen equivalences between the model structures for $n$-truncated quasi-categories and Rezk's $(n,1)$-$\Theta$-spaces \cite{MR2578310}, which are another model for $(n,1)$-categories. We also prove (Theorem \ref{cat1qcat}) that the nerve functor $N \colon \Cat \lra \sSet$ is the right adjoint of a Quillen equivalence between the folk model structure for categories and the model structure for $1$-truncated quasi-categories, and hence (Theorem \ref{thmqrezk}) that the composite adjunction 
\begin{equation*} 
\xymatrix{
\Cat \ar@<-1.5ex>[rr]^-{\hdash}_-N && \ar@<-1.5ex>[ll]_-{\tau_1} [\Delta^\mathrm{op},\Set] \ar@<-1.5ex>[rr]^-{\hdash}_-{t^!} && \ar@<-1.5ex>[ll]_-{t_!} [(\Delta\times\Delta)^\mathrm{op},\Set],
}
\end{equation*}
whose right adjoint is Rezk's ``classifying diagram'' functor \cite{MR1804411},
 is a Quillen equivalence between the model structures for categories and Rezk's $(1,1)$-$\Theta$-spaces.

The need for the $n=1$ case of these results arose during the first-named author's work on the paper \cite{campbellnerve}, wherein they serve as  part of the proofs that certain adjunctions
\begin{equation*} 
\xymatrix{
\Bicat_\mathrm{s} \ar@<-1.5ex>[rr]^-{\hdash}_-N && \ar@<-1.5ex>[ll]_-{\tau_b} [\Theta_2^\mathrm{op},\Set] \ar@<-1.5ex>[rr]^-{\hdash}_-{t^!} && \ar@<-1.5ex>[ll]_-{t_!} [(\Theta_2\times\Delta)^\mathrm{op},\Set]
}
\end{equation*}
are Quillen equivalences between Lack's model structure for bicategories \cite{MR2138540}, the Bousfield localisation of Ara's model structure for $2$-quasi-categories \cite{MR3350089} with respect to the boundary inclusion $\partial\Theta_2[1;3] \lra \Theta_2[1;3]$, and Rezk's model structure for $(2,2)$-$\Theta$-spaces \cite{MR2578310}.

We begin this paper in  \S\ref{secsimp} with a collection of some preliminary notions and results pertaining to simplicial sets and $n$-types. Our study of $n$-truncated quasi-categories begins in \S\ref{sectruncated}, where we construct the model structure for  $n$-truncated quasi-categories, and continues in \S\ref{seccatnequiv}, where we characterise the weak equivalences of this model structure. 
Finally, in \S \ref{secquillen} we  prove the aforementioned Quillen equivalences between the model categories of categories and $1$-truncated quasi-categories and between the model categories of $n$-truncated quasi-categories and Rezk's $(n,1)$-$\Theta$-spaces. In an appendix \S \ref{secbousfield}, we recall some of the basic theory of Bousfield localisations of model categories, including two criteria for detecting Quillen equivalences between Bousfield localisations.

\subsection*{Acknowledgements} 
The first-named author gratefully acknowledges the support of Australian Research Council Discovery Project DP160101519 and Future Fellowship FT160100393. The second-named author gratefully acknowledges the support of a Macquarie University iMQRes PhD scholarship.

\section{Simplicial preliminaries} \label{secsimp}
In this section, we collect some preliminary notions and results pertaining to simplicial sets and homotopy $n$-types (as modelled by Kan complexes) that we will use in the following sections on truncated quasi-categories.  For further background on simplicial sets, see for example \cite{MR0210125}, \cite{MR1711612}, and \cite[Chapitre 2]{MR2294028}.

We begin with the definition of (homotopy) $n$-types, which we will use in the definition of truncated quasi-categories in \S\ref{sectruncated}.

\begin{definition}
Let $n \geq 0$ be an integer. A Kan complex $X$ is said to be an \emph{$n$-type} if, for each object (i.e.\ $0$-simplex)  $x \in X_0$ and each integer $m > n$, the homotopy group $\pi_m(X,x)$ is trivial (i.e.\ $\pi_m(X,x) \cong 1$).
\end{definition}

\begin{example} \label{ex0type}
Every discrete (i.e.\ constant) simplicial set is a $0$-type. Furthermore, a Kan complex $X$ is a $0$-type if and only if the unit morphism $X \lra \mathrm{disc}(\pi_0X)$ of the adjunction
\begin{equation} \label{0typeadj}
\xymatrix{
\Set \ar@<-1.5ex>[rr]^-{\hdash}_-{\mathrm{disc}} && \ar@<-1.5ex>[ll]_-{\pi_0} \sSet
}
\end{equation}
 is a homotopy equivalence. 
\end{example}

It is natural to extend the notion of $n$-type to lower values of $n$ as follows. Recall that a Kan complex $X$ is said to be \emph{contractible} if the unique morphism $X \lra \Delta^0$ is a homotopy equivalence.

\begin{definition} A Kan complex is said to be a \emph{$(-1)$-type} if it is either empty or contractible, and is said to be a \emph{$(-2)$-type} if it is contractible.
\end{definition}

In our study of truncated quasi-categories, we will use the following well-known alternative characterisation of $n$-types in terms of a lifting property (whose proof is a standard exercise).

\begin{proposition} \label{kantype}
Let $n \geq -2$ be an integer. A Kan complex is an $n$-type if and only if it has the right lifting property with respect to the boundary inclusion $\partial\Delta^m \lra \Delta^m$ for every $m \geq n+2$.
\end{proposition}

We will see in Proposition \ref{truncisacyc} that $n$-truncated quasi-categories can be characterised by the same lifting property. For this reason, we now  record this lifting property in the following definition and explore some of its consequences.

\begin{definition} \label{acycdef}
Let $n \geq -1$ be an integer. A simplicial set $X$ is said to be \emph{$n$-acyclic} if it has the right lifting property with respect to the boundary inclusion $\partial\Delta^m \lra \Delta^m$ for every $m >n$.
\end{definition}

In these terms, Proposition \ref{kantype} states that, for every $n \geq -2$, a Kan complex is an $n$-type if and only if it is $(n+1)$-acyclic. Similarly, we will prove in Proposition \ref{truncisacyc} that, for every $n \geq -1$, a quasi-category is $n$-truncated if and only if it is $(n+1)$-acyclic. 
This lifting property will be very useful, as it yields a large class of morphisms with respect to which $n$-truncated quasi-categories have the right lifting property.

\begin{definition} \label{nbijdef}
Let $n \geq -1$ be an integer.
A morphism of simplicial sets $f \colon X \lra Y$ is said to be \emph{$n$-bijective} if the function $f_k \colon X_k \lra Y_k$ is bijective for each $0 \leq k \leq n$. 
\end{definition}

\begin{lemma} \label{satcor}
Let $n \geq -1$ be an integer. A simplicial set is $n$-acyclic if and only if it has the right lifting property with respect to every $n$-bijective monomorphism of simplicial sets.
\end{lemma}
\begin{proof}
Since the boundary inclusion $\partial\Delta^m \lra \Delta^m$ is an $n$-bijective monomorphism for every $m > n$, any simplicial set with the stated lifting property is $n$-acyclic. Note that any class of morphisms defined by a left lifting property is stable under pushout and closed under coproducts and countable composition. The converse then follows from the fact that any $n$-bijective monomorphism can be decomposed into a countable composite of pushouts of coproducts of the boundary inclusions $\partial\Delta^m \lra \Delta^m$ for $m > n$, as in \cite[\S II.3.8]{MR0210125}.
\end{proof}

\begin{remark}
If $n = -1$, the condition in Definition \ref{nbijdef} is vacuous, and so every morphism of simplicial sets is $(-1)$-bijective.
Hence the $n =-1$ case of Lemma \ref{satcor} states that a simplicial set is $(-1)$-acyclic if and only if it is an injective object in the category of simplicial sets, i.e.\ a contractible Kan complex, i.e.\ a $(-2)$-type. Furthermore, a simplicial set is $0$-acyclic if and only if it is a $(-1)$-type.
\end{remark}

In \S\ref{sectruncated}, we will use the following consequence of Lemma \ref{satcor} to prove that the model structures for $n$-truncated quasi-categories are cartesian.

\begin{lemma} \label{acyccart} Let $n \geq -1$ be an integer.
For every simplicial set $A$ and $n$-acyclic simplicial set $X$,  the internal hom simplicial set $X^A$ is $n$-acyclic.
\end{lemma}
\begin{proof}
It is required to prove that $X^A$ has the right lifting property with respect to the boundary inclusion $b_m \colon \partial\Delta^m \lra \Delta^m$ for every $m > n$. By adjunction, this is so if and only if $X$ has the right lifting property with respect to the morphism $b_m \times A \colon \partial\Delta^m \times A \lra \Delta^m \times A$ for every $m > n$. But the $n$-acyclic simplicial set $X$ has this  lifting property by Lemma \ref{satcor}, since the morphism $b_m \times A$ is an $n$-bijective monomorphism for every $m > n$. 
\end{proof}

The following lemma shows that the property of $n$-acyclicity can be understood as a weakening of the property of $n$-coskeletality. (Recall that a simplicial set $X$ is said to be \emph{$n$-coskeletal} if the unit morphism $X \lra \cosk_nX$ to its $n$-coskeleton is an isomorphism; dually, $X$ is said to be \emph{$n$-skeletal} if the counit morphism $\sk_nX \lra X$ from its $n$-skeleton is an isomorphism.)

\begin{lemma} \label{acyccosk} Let $n \geq -1$ be an integer.
A simplicial set $X$ is $n$-acyclic if and only if the unit morphism $X \lra \cosk_nX$ is a trivial fibration. 
\end{lemma}
\begin{proof}
Let $X$ be a simplicial set. By definition, the unit morphism $X \lra \cosk_nX$ is a trivial fibration if and only if it has the right lifting property with respect to the boundary inclusion $\partial\Delta^m \lra \Delta^m$ for each $m \geq 0$. By adjointness, this is so if and only if $X$ has the right lifting property with respect to the inclusion $\sk_n\Delta^m \cup \partial\Delta^m \lra \Delta^m$ for each $m \geq 0$. If $m \leq n$, this inclusion is an identity, and so the lifting property is satisfied trivially. If $m >n$, this inclusion is  the boundary inclusion $\partial\Delta^m \lra \Delta^m$. Thus the two properties in the statement are seen to be equivalent.
\end{proof}

\begin{remark}
By an argument similar to the proof of Lemma \ref{satcor}, one can show that a simplicial set is $n$-coskeletal if and only if it has the \emph{unique} right lifting property with respect to the boundary inclusion $\partial\Delta^m \lra \Delta^m$ for each $m > n$. This gives another sense in which the property of $n$-acyclicity is a weakening of the property of $n$-coskeletality.
\end{remark}

Now, recall that (the simplicial analogue of) Whitehead's theorem states that a morphism of Kan complexes $f \colon X \lra Y$ is a homotopy equivalence if and only if (i) the induced function $\pi_0(f) \colon \pi_0X \lra \pi_0Y$ is a bijection and (ii) for every integer $n \geq 1$ and every object $x$ of $X$, the induced function $\pi_n(f) \colon \pi_n(X,x) \lra \pi_n(Y,fx)$ is a bijection (and hence an isomorphism of groups). 
 We will use the following weakenings of these properties in our characterisation of the weak equivalences in the model structures for $n$-truncated quasi-categories in \S\ref{seccatnequiv}.

\begin{definition}
Let $n \geq 0$ be an integer. A morphism of Kan complexes $f \colon X \lra Y$ is said to be a \emph{homotopy $n$-equivalence} if
\begin{enumerate}[leftmargin=*,label=(\roman*)]
\item the induced function $\pi_0(f) \colon \pi_0X \lra \pi_0Y$ is a bijection, and
\item  the induced function $\pi_k(f) \colon \pi_k(X,x) \lra \pi_k(Y,fx)$ is a bijection (and hence an isomorphism of groups) for every integer $1 \leq k \leq n$ and every object $x \in X$.
\end{enumerate}
\end{definition}

Thus a morphism of Kan complexes is a homotopy $0$-equivalence if and only if it is inverted by the functor $\pi_0 \colon\sSet \lra \Set$. We similarly define a morphism of Kan complexes to be a \emph{homotopy $(-1)$-equivalence} if it is inverted by the functor $\pi_{-1} \colon \sSet \lra \{0<1\}$ that sends the empty simplicial set to $0$ and every nonempty simplicial set to $1$. Thus a morphism of Kan complexes is a homotopy $(-1)$-equivalence if either (i) its domain and codomain are both empty, or (ii) its domain and codomain are both nonempty. 
Furthermore, we define a morphism of Kan complexes to be a \emph{homotopy $(-2)$-equivalence} if it is inverted by the unique functor $\pi_{-2} \colon \sSet \lra \mathbf{1}$ to the terminal category; thus every morphism of Kan complexes is a homotopy $(-2)$-equivalence.

For each $n \geq -2$, a morphism of $n$-types is a homotopy equivalence if and only if it is a homotopy $n$-equivalence: if $n \geq 0$, this follows from Whitehead's theorem; if $n = -2,-1$, this follows from the fact that any morphism between contractible Kan complexes is a homotopy equivalence.

\begin{remark} \label{ntypemodstr}
It is a standard result (cf.\ \cite[\S1.5]{MR1944041} and \cite[\S 9.2]{MR2294028}) that, for each integer $n \geq -2$, the $n$-types are the fibrant objects of the Bousfield localisation of the model structure for Kan complexes with respect to the boundary inclusion $\partial\Delta^{n+2} \lra \Delta^{n+2}$, and that  a morphism of Kan complexes is a weak equivalence in this Bousfield localisation if and only if it is a homotopy $n$-equivalence in the sense of the above definitions. In \S\S\ref{sectruncated}--\ref{seccatnequiv}, we will generalise both of these statements to $n$-truncated quasi-categories.
\end{remark}

We will use the following two properties of the class of homotopy $n$-equivalences in \S\ref{seccatnequiv}. 

\begin{lemma} \label{homo23}
Let $n \geq -2$ be an integer and let $f \colon X \lra Y$ and $g \colon Y \lra Z$ be morphisms of Kan complexes. If two of the morphisms $f,g,gf$ are homotopy $n$-equivalences, then so is the third.
\end{lemma}
\begin{proof}
This is proved by any of the standard arguments proving that the class of morphisms of Kan complexes described in the statement of Whitehead's theorem enjoys the same property.
\end{proof}

\begin{lemma} \label{bijimpho}
Let $n \geq -2$ be an integer. An $(n+1)$-bijective morphism of Kan complexes is a homotopy $n$-equivalence.
\end{lemma}
\begin{proof}
The cases $n = -2,-1$ are immediate. Suppose $n \geq 0$. The result follows from the facts that the set of connected components of a simplicial set depends only on its $1$-skeleton, and that, for each integer $k \geq 1$, the $k$th homotopy groups of a Kan complex depend only on its $(k+1)$-skeleton (since their elements are pointed homotopy classes of morphisms to $X$ from the (simplicial) $k$-sphere, whose homotopy type can be modelled by a $k$-skeletal simplicial set, e.g.\ $\Delta^k/\partial\Delta^k$ or $\partial\Delta^{k+1}$).
\end{proof}

\section{Truncated quasi-categories} \label{sectruncated}

\emph{Throughout this section, let $n \geq -1$ be an integer.}

\medskip

\begin{remark}
As mentioned in \S1, some of the results of \S\S\ref{sectruncated}--\ref{secquillen} are stated without proof in Joyal's notes \cite[\S26]{joyalnotes}. These results will be indicated below by references to the numbered paragraphs of those notes in which they are stated.  
(Given that one of the purposes of this paper is to provide proofs for these statements, we beg the reader's patience if we spell out the occasional ``obvious'' argument.)
\end{remark}

As recalled in Remark \ref{ntypemodstr}, the (homotopy) $n$-types are the fibrant objects of the Bousfield localisation of the model structure for Kan complexes with respect to the boundary inclusion $\partial\Delta^{n+2} \lra \Delta^{n+2}$, and a morphism of Kan complexes is a weak equivalence in this Bousfield localisation if and only if it is a homotopy $n$-equivalence. The goal of this section and the next is to prove the analogous results for quasi-categories. In this section, we prove that the $n$-truncated  quasi-categories are the fibrant objects of the Bousfield localisation of Joyal's model structure  for quasi-categories with respect to the same boundary inclusion (Theorem \ref{newthm}).  In \S\ref{seccatnequiv}, we will prove that a morphism of quasi-categories is a weak equivalence in this Bousfield localisation if and only if it is a categorical $n$-equivalence (Theorem \ref{wethm}). (Note that the first of these two results is new to this paper, whereas the second was stated without proof in \cite[\S26.6]{joyalnotes}.)

We refer the reader to Appendix \ref{secbousfield} for the necessary background on Bousfield localisations, and to \cite{MR1935979}, \cite[\S1]{MR2342834}, and \cite[Chapter 1]{MR2522659} for a more than sufficient background in the theory of quasi-categories. In particular, recall that there is a (left proper and combinatorial) cartesian model structure due to Joyal on the category of simplicial sets whose cofibrations are the monomorphisms and whose fibrant objects are the quasi-categories \cite[Theorem 6.12]{joyalbarcelona}. We call this model structure the \emph{model structure for quasi-categories}; the weak equivalences and fibrations between fibrant objects of this model structure will be called \emph{weak categorical equivalences} and \emph{isofibrations} respectively. (Note that, following \cite{MR2764043}, we will sometimes denote the category of simplicial sets equipped with the model structures  for Kan complexes and quasi-categories by $\sSet_\mathrm{K}$ and $\sSet_\mathrm{J}$ respectively.)

To begin, let us recall the definition of the hom-spaces of a quasi-category. For each pair of objects (i.e.\ $0$-simplices) $x,y$ of a quasi-category $X$, their \emph{hom-space} $\Hom_X(x,y)$ is the Kan complex defined by the pullback
 \begin{equation} \label{hompb}
 \cd{
 \Hom_X(x,y) \ar[r] \ar[d] \fatpullbackcorner & X^{\Delta^1} \ar[d]^-{(X^{\delta^1},X^{\delta^0})} \\
 \Delta^0 \ar[r]_-{(x,y)} & X \times X
 }
 \end{equation}
 in the category of simplicial sets.  By \cite[Proposition 4.5]{MR2764043}, this hom-space construction defines the right adjoint of a Quillen adjunction
\begin{equation} \label{susphomadj}
 \xymatrix{
\partial\Delta^1\backslash\sSet_\mathrm{J} \ar@<-1.5ex>[rr]^-{\hdash}_-{\Hom} && \ar@<-1.5ex>[ll]_-{\Sigma} \sSet_\mathrm{K}
}
\end{equation}
 between the category of bipointed simplicial sets (note that $\partial\Delta^1 \cong \Delta^0 + \Delta^0$) equipped with the model structure induced by the model structure for quasi-categories and the category of simplicial sets equipped with the model structure for Kan complexes, whose left adjoint sends a simplicial set $U$ to its \emph{(two-point) suspension} $\Sigma U$, defined by the pushout  \begin{equation} \label{susppo}
 \cd{
U \times \partial\Delta^1 \ar[r]^-{\mathrm{pr}_2} \ar[d] & \partial\Delta^1 \ar[d]^-{(\bot,\top)} \\
U \times \Delta^1 \ar[r] & \fatpushoutcorner \Sigma U
 }
 \end{equation}
 in the category of simplicial sets; note that the simplicial set $\Sigma U$ has precisely two $0$-simplices, which we denote by $\bot$ and $\top$, as in the diagram above. 
 
Next, recall that one can assign to each category $A$ a quasi-category $NA$ via the nerve functor $N \colon \Cat \lra \sSet$, which defines the fully faithful right adjoint of an adjunction
\begin{equation} \label{nerveadj}
\xymatrix{
\Cat \ar@<-1.5ex>[rr]^-{\hdash}_-{N} && \ar@<-1.5ex>[ll]_-{\tau_1} \sSet
}
\end{equation}
whose left adjoint sends a simplicial set $X$ to its fundamental category $\tau_1X$ (see \cite[\S II.4]{MR0210125}). If $X$ is a quasi-category, then its fundamental category $\tau_1X$ is isomorphic to its \emph{homotopy category} $\ho X$, which was first constructed by Boardman and Vogt \cite[\S IV.2]{MR0420609} (for a detailed proof, see \cite[Chapter 1]{joyalbarcelona}).
The homotopy category $\ho X$ of a quasi-category $X$ has the same set of objects as $X$, and its hom-sets $(\ho X)(x,y) \cong  \pi_0(\Hom_X(x,y))$ are isomorphic to the sets of connected components of the hom-spaces of $X$; thus the unit morphism $X \lra N(\ho X)$ of the adjunction (\ref{nerveadj}) is a bijection on objects, and is given on hom-spaces by the unit morphism $\Hom_X(x,y) \lra \mathrm{disc}(\pi_0(\Hom_X(x,y)))$ of the adjunction $\pi_0 \dashv \mathrm{disc}$ (\ref{0typeadj}). A morphism (i.e.\ a $1$-simplex) in a quasi-category $X$ is said to be an \emph{isomorphism} if it is sent by the unit morphism $X \lra N(\ho X)$ to an isomorphism in $\ho X$.

 A morphism of quasi-categories $f \colon X \lra Y$ is said to be \emph{essentially surjective on objects} if the induced functor between homotopy categories $\ho(f) \colon \ho X \lra \ho Y$ is essentially surjective on objects. 
 A fundamental theorem of quasi-category theory  states that a morphism of quasi-categories $f \colon X \lra Y$ is an equivalence of quasi-categories (i.e.\ a weak categorical equivalence between quasi-categories) if and only if it is essentially surjective on objects and a homotopy equivalence on hom-spaces, that is, for each pair of objects $x,y \in X$, the induced morphism of hom-spaces $f \colon \Hom_X(x,y) \lra \Hom_Y(fx,fy)$ is a homotopy equivalence of Kan complexes. 

We now recall the definition of $n$-truncated quasi-categories from \cite[\S26]{joyalnotes}. 

\begin{definition} \label{deftrunc}
A quasi-category $X$ is said to be \emph{$n$-truncated} if, for each pair of objects $x,y \in X$, the hom-space $\Hom_X(x,y)$ is an $(n-1)$-type.
\end{definition}

\begin{remark}
In \cite[Proposition 2.3.4.18]{MR2522659}, Lurie proved that a quasi-category is $n$-truncated if and only if it is equivalent to an \emph{$n$-category} in the sense of \cite[Definition 2.3.4.1]{MR2522659}. We will not use this result in the present paper.
\end{remark}

Before proceeding with the study of the homotopy theory of $n$-truncated quasi-categories, let us examine the low dimensional cases of this definition. 
By definition, a quasi-category is $1$-truncated if and only if its hom-spaces are $0$-types. For example, the nerve $NA$ of a category $A$ is a $1$-truncated quasi-category, since its hom-spaces are the discrete simplicial sets $\Hom_{NA}(a,b) \cong \mathrm{disc}(A(a,b))$ given by the hom-sets of $A$, and since every discrete simplicial set is a $0$-type.

\begin{proposition}[{\cite[\S26.1]{joyalnotes}}] \label{catprop}
A quasi-category $X$ is $1$-truncated if and only if the unit morphism $X \lra N(\ho X)$ is an equivalence of quasi-categories. In particular, the nerve of a category is a $1$-truncated quasi-category. 
\end{proposition}
\begin{proof}
Let $X$ be a quasi-category. By construction, the  unit morphism $X \lra N(\ho X)$ is bijective on objects, and therefore is an equivalence if and only if it is a homotopy equivalence on hom-spaces, that is, if and only if the unit morphism  $\Hom_X(x,y) \lra \mathrm{disc}(\pi_0(\Hom_X(x,y)))$ is a homotopy equivalence for each pair of objects $x,y \in X$. But this is so precisely when each hom-space $\Hom_X(x,y)$ is a $0$-type (see Example \ref{ex0type}), that is, precisely when $X$ is $1$-truncated.
\end{proof}

\begin{remark}
For any quasi-category $X$, the unit morphism $X\lra N(\ho X)$ is an isofibration. Hence a quasi-category $X$ is $1$-truncated if and only if the unit morphism $X \lra N(\ho X)$ is a trivial fibration.
\end{remark}

Recall that a category is a \emph{preorder} if each of its hom-sets has at most one element. A category is a preorder if and only if it is equivalent to a poset (partially ordered set): the quotient of a preorder by the congruence $x \sim y$ $\iff$ $x \leq y$ \& $y \leq x$ defines an equivalent poset, which we call its \emph{poset quotient}; conversely, any category equivalent to a preorder is evidently a preorder, and a poset is in particular a preorder. 

\begin{proposition}[{\cite[\S26.2]{joyalnotes}}] \label{posetprop}
A quasi-category is $0$-truncated if and only if it is $1$-truncated and its homotopy category is equivalent to a poset. In particular, the nerve of a preorder is a $0$-truncated quasi-category.
\end{proposition}
\begin{proof}
A Kan complex is a $(-1)$-type if and only if it is a $0$-type and its set of connected components has at most one element. Hence a quasi-category $X$ is $0$-truncated if and only if it is $1$-truncated and its homotopy category is a preorder, that is, equivalent to a poset.
\end{proof}

A quasi-category is $(-1)$-truncated if and only if it is empty or a contractible Kan complex, that is, if and only if it is a $(-1)$-type: if $X$ is a nonempty $(-1)$-truncated quasi-category, then its hom-spaces are contractible, and so the unique morphism $X \lra \Delta^0$ is surjective on objects and a homotopy equivalence on hom-spaces, and is thus an equivalence of quasi-categories, and hence a trivial fibration. Similarly, one could define a quasi-category to be $(-2)$-truncated if it is a $(-2)$-type, i.e.\ a contractible Kan complex.

We now proceed towards the main goal of this section, which is to prove that the $n$-truncated quasi-categories are the fibrant objects of the Bousfield localisation of the model structure for quasi-categories with respect to the boundary inclusion $\partial\Delta^{n+2} \lra \Delta^{n+2}$.   Our first step will be to show that $n$-truncated quasi-categories can be characterised in terms of a lifting property. To this end,  it will be convenient to use an alternative model for the hom-spaces of a quasi-category.

Recall that a morphism of simplicial sets $f \colon X \lra Y$ is said to be a \emph{right fibration} if it has the right lifting property with respect to the horn inclusion $\Lambda^m_k \lra \Delta^m$ for every $m \geq 1$ and $0 < k \leq m$ (see \cite[\S2]{MR1935979} or \cite[Chapter 2]{MR2522659}). For each object $x$ of a quasi-category $X$, one obtains by the join and slice constructions of \cite[\S3]{MR1935979} a right fibration $X/x \lra X$ whose domain is the \emph{slice quasi-category} $X/x$ (see \cite[\S 1.2.9]{MR2522659}). The slice quasi-category construction defines the right adjoint of an adjunction
\begin{equation*}
 \xymatrix{
\Delta^0\backslash\sSet \ar@<-1.5ex>[rr]^-{\hdash}_-{\mathrm{slice}} && \ar@<-1.5ex>[ll]_-{-\star \Delta^0} \sSet
}
\end{equation*}
whose left adjoint sends a simplicial set $U$ to the \emph{right cone} of $U$, i.e.\ the join $U \star \Delta^0$ with base point $\Delta^0 \cong \emptyset \star \Delta^0 \lra U \star \Delta^0$. Thus, for each $k \geq 0$, a $k$-simplex of the slice quasi-category $X/x$ is given by a $(k+1)$-simplex of $X$ whose final vertex is $x$; the right fibration $X/x \lra X$ sends a $k$-simplex of $X/x$ to the face opposite the last vertex of the corresponding $(k+1)$-simplex of $X$. (See \cite[\S3]{MR1935979} and \cite[\S\S 1.2.8--9]{MR2522659} for further details.)

For each pair of objects $x,y$ of a quasi-category $X$, the \emph{right hom-space} $\Hom_X^R(x,y)$ is defined by the pullback
\begin{equation*}
\cd{
\Hom_X^R(x,y) \fatpullbackcorner \ar[r] \ar[d] & X/y \ar[d] \\
\Delta^0 \ar[r]_-x & X
}
\end{equation*}
in the category of simplicial sets \cite[\S1.2.2]{MR2522659}. Since the projection $X/y \lra X$ is a right fibration, it follows that the right hom-space $\Hom_X^R(x,y)$ is a Kan complex (see \cite[Proposition 1.2.2.3]{MR2522659}). A $k$-simplex of $\Hom_X^R(x,y)$ is given by a $(k+1)$-simplex  of $X$ whose last vertex is $y$ and whose face opposite the last vertex is the degenerate $k$-simplex on $x$. 

Importantly, for each pair of objects $x,y$ of a quasi-category $X$, there is a homotopy equivalence $\Hom_X^R(x,y) \simeq \Hom_X(x,y)$ between the right hom-space and the hom-space  (see \cite[Corollary 4.2.1.8]{MR2522659}). Hence a quasi-category is $n$-truncated if and only if each of its right hom-spaces is an $(n-1)$-type.

The characterisation of $n$-truncated quasi-categories in terms of a lifting property depends on the following lemma. Recall from Definition \ref{acycdef} that a simplicial set is said to be $n$-acyclic if it has the right lifting property with respect to the boundary inclusion $\partial\Delta^m \lra \Delta^m$ for every $m > n$.

\begin{lemma} \label{rightfibre}
Let $f \colon X \lra Y$ be a  right fibration of simplicial sets. Then the following properties are equivalent.
\begin{enumerate}[leftmargin=*, font=\normalfont, label=(\roman*)]
\item  $f$
 has the right lifting property with respect to the boundary inclusion $\partial\Delta^m \lra \Delta^m$ for every $m > n$.
 \item For every $0$-simplex $y \in Y_0$, the fibre $f^{-1}(y)$ is an $(n-1)$-type.
 \item For every $0$-simplex $y \in Y_0$, the fibre $f^{-1}(y)$ is $n$-acyclic.
 \end{enumerate}
\end{lemma}
\begin{proof}
Since the fibres of a right fibration are Kan complexes, the equivalence (ii) $\iff$ (iii) follows from Proposition \ref{kantype}. Furthermore, since any pullback of a morphism satisfying the lifting property (i) inherits this lifting property, we have the implication (i) $\implies$ (iii). 

It remains to prove the implication (iii) $\implies$ (i).
If $n=-1$, this implication is precisely \cite[Lemma 2.1.3.4]{MR2522659}, which states that a right fibration whose fibres are contractible is a trivial fibration. In fact, the proof of the cited result proves moreover that, for each $k \geq 0$, if the fibres of a right fibration each have the right lifting property with respect to the boundary inclusion $\partial\Delta^k \lra \Delta^k$, then  the right fibration also has the right lifting property with respect to that boundary inclusion. This proves the implication (iii) $\implies$ (i) for an arbitrary $n \geq -1$.
\end{proof}

By applying Lemma \ref{rightfibre} to the right fibrations of the form $X/x \lra X$,  we can characterise the $n$-truncated quasi-categories by the following lifting property. 

\begin{proposition}[{\cite[\S\S26.1--3]{joyalnotes}}] \label{truncisacyc}
A quasi-category is $n$-truncated if and only if it has the right lifting property with respect to the boundary inclusion $\partial\Delta^m \lra \Delta^m$ for every $m \geq  n+2$.
\end{proposition}
\begin{proof}
By the homotopy equivalences between the hom-spaces and the right hom-spaces of a quasi-category \cite[Corollary 4.2.1.8]{MR2522659}, a quasi-category $X$ is $n$-truncated if and only if the right hom-space $\Hom_X^R(x,y)$ is an $(n-1)$-type for each pair of objects $x,y \in X$. We thus have that a quasi-category $X$ is $n$-truncated if and only if every fibre of the right fibration $X/y \lra X$ is an $(n-1)$-type for every object $y\in X$. By Lemma \ref{rightfibre}, this is so if and only if the right fibration $X/y \lra X$ has the right lifting property with respect to the boundary inclusion $\partial\Delta^m \lra \Delta^m$ for every $m > n$ and every $y \in X$. By adjointness (see \cite[Lemma 3.6]{MR1935979}), this lifting property is satisfied if and only if  $X$ has the right lifting property with respect to the pushout-join $(\partial\Delta^m \star \Delta^0) \cup (\Delta^m \star \emptyset) \lra \Delta^m \star \Delta^0$ for every $m > n$. But  this pushout-join is none other than the boundary inclusion $\partial\Delta^{m+1} \lra \Delta^{m+1}$ \cite[Lemma 3.3]{MR1935979}. Hence we have shown that a quasi-category $X$ is $n$-truncated if and only if it has the right lifting property with respect to the boundary inclusion $\partial\Delta^{m+1} \lra \Delta^{m+1}$ for every $m > n$, as required.
\end{proof}

In the terminology of Definition \ref{acycdef}, Proposition \ref{truncisacyc} states that a quasi-category is $n$-truncated if and only if it is $(n+1)$-acyclic. Thus we may deduce that the class of $n$-truncated quasi-categories inherits the following properties from the class of $(n+1)$-acyclic simplicial sets.

\begin{corollary} \label{acyclift}
A quasi-category is $n$-truncated if and only if it has the right lifting property with respect to every $(n+1)$-bijective monomorphism of simplicial sets.
\end{corollary}
\begin{proof}
This is a consequence of Proposition \ref{truncisacyc} and Lemma \ref{satcor}.
\end{proof}

\begin{corollary} \label{cartpropqcat}
For every simplicial set $A$ and $n$-truncated quasi-category $X$,  the internal hom simplicial set $X^A$ is an $n$-truncated quasi-category.
\end{corollary}
\begin{proof}
We have by \cite[Corollary 2.19]{joyalbarcelona} that $X^A$ is a quasi-category. Hence by Proposition \ref{truncisacyc}, $X^A$ is an $n$-truncated quasi-category if and only if it is $(n+1)$-acyclic. The result then follows from Corollary \ref{acyccart}.
\end{proof}

\begin{remark}
The result of Corollary \ref{cartpropqcat} was proved by Lurie as \cite[Corollary 2.3.4.20]{MR2522659}. Our proof of this result is more direct and elementary than Lurie's proof, which uses the corresponding result for $n$-categories (in his sense) \cite[Proposition 2.3.4.8]{MR2522659} and the fact that any quasi-category is equivalent to a minimal quasi-category \cite[Proposition 2.3.3.8]{MR2522659}.
\end{remark}

\begin{corollary} \label{qunittrivfib}
A quasi-category $X$ is $n$-truncated if and only if the unit morphism $X \lra \cosk_{n+1}X$ is a trivial fibration.
\end{corollary}
\begin{proof}
This is a consequence of Proposition \ref{truncisacyc} and Lemma \ref{acyccosk}.
\end{proof}

In Propositions \ref{kantype} and \ref{truncisacyc}, $n$-types and $n$-truncated quasi-categories were both characterised by the same lifting property. Hence we may deduce the following corollary.

\begin{corollary}[{\cite[\S\S26.1--3]{joyalnotes}}] \label{kancor}
A Kan complex is an $n$-truncated quasi-category if and only if it is an $n$-type.
\end{corollary}
\begin{proof}
By definition, every Kan complex is a quasi-category. Hence by Proposition \ref{truncisacyc}, a Kan complex is an $n$-truncated quasi-category if and only if it is $(n+1)$-acyclic, which is so, by Proposition \ref{kantype}, precisely when it is an $n$-type.
\end{proof}

Next, we deduce from Proposition \ref{truncisacyc} a further characterisation of $n$-truncated quasi-categories  as the quasi-categories that are local (see (\ref{precomploc})) with respect to the boundary inclusion $\partial\Delta^{n+2} \lra \Delta^{n+2}$ in the model structure $\sSet_\mathrm{J}$ for quasi-categories. As explained in Appendix \ref{secbousfield}, this will require a model for the derived hom-spaces of the model category $\sSet_\mathrm{J}$, which we will obtain by Lemma \ref{recoghom} from the  Quillen adjunction (\ref{joyalresadj}) below. 

Let $\mathbf{qCat}$ and $\mathbf{Kan}$ denote the full subcategories of $\sSet$ consisting of the quasi-categories and the Kan complexes respectively. By \cite[Theorem 4.19]{joyalbarcelona}, the full inclusion $\mathbf{Kan} \lra \mathbf{qCat}$ has a right adjoint $J \colon \mathbf{qCat} \lra \mathbf{Kan}$, which sends a quasi-category $X$ to its maximal sub Kan complex $J(X)$. By \cite[Lemma 4.18]{joyalbarcelona}, a simplex of $X$ belongs to the simplicial subset $J(X)$ if and only if each of its 1-simplices is an isomorphism in $X$. 
Note that, by \cite[Proposition 4.27]{joyalbarcelona}, the functor $J$ sends isofibrations to Kan fibrations.

Let $X$ be a quasi-category.  By \cite[Corollary 5.11]{joyalbarcelona}, there is an adjunction
\begin{equation} \label{joyalresadj}
\xymatrix{
\sSet_\mathrm{J}^\mathrm{op} \ar@<-1.5ex>[rr]^-{\hdash}_-{J(X^-)} && \sSet_\mathrm{K} \ar@<-1.5ex>[ll]_-{X^{(-)}}
}
\end{equation}
whose right adjoint sends a simplicial set $A$ to the Kan complex $J(X^A)$, and whose left adjoint sends a simplicial set $U$ to the  full sub-quasi-category $X^{(U)}$ of $X^U$ consisting of the morphisms of simplicial sets $U \lra X$ which factor through $J(X)$, i.e.\ which send each $1$-simplex of $U$ to an isomorphism in $X$. Moreover, by \cite[Theorems 5.7, 5.10]{joyalbarcelona}, this adjunction is a Quillen adjunction between (the opposite of) the model structure $\sSet_\mathrm{J}$ for quasi-categories and the model structure $\sSet_\mathrm{K}$ for Kan complexes as indicated. 

Hence, for each simplicial set $A$ and quasi-category $X$, Lemma \ref{recoghom} applied to the Quillen adjunction (\ref{joyalresadj}) implies that the Kan complex $J(X^A)$ is a model for the derived hom-space $\underline{\Ho\sSet_\mathrm{J}}(A,X)$ from $A$ to $X$ in the model structure for quasi-categories. We may therefore deduce the following lemma.

\begin{lemma} \label{qcatloclem}
A quasi-category $X$ is local with respect to a morphism $f \colon A \lra B$ in the model structure for quasi-categories if and only if the morphism $J(X^f) \colon J(X^B) \lra J(X^A)$ is a homotopy equivalence of Kan complexes.
\end{lemma}
\begin{proof}
By definition (see Appendix \ref{secbousfield}), a quasi-category $X$ is local with respect to a morphism $f \colon A \lra B$ in the model category $\sSet_\mathrm{J}$ if and only if this morphism is sent to an isomorphism by the functor $$\underline{\Ho\sSet_\mathrm{J}}(-,X) : \Ho\sSet_\mathrm{J}^\mathrm{op} \lra \mathscr{H}.$$ Since $X^{(\Delta^0)} \cong X$, Lemma  \ref{recoghom} implies that this functor is naturally isomorphic to the derived right adjoint of the Quillen adjunction (\ref{joyalresadj}). Therefore, since every object of $\sSet_\mathrm{J}$ is cofibrant, a morphism of simplicial sets is sent to an isomorphism by the functor $\underline{\Ho\sSet_\mathrm{J}}(-,X)$ if and only if it is sent to a homotopy equivalence of Kan complexes by the right Quillen functor $J(X^-) \colon \sSet_\mathrm{J}^\mathrm{op} \lra \sSet_\mathrm{K}$,  as required.
\end{proof}

\begin{remark} \label{howkanloc}
A Kan complex is local with respect to a given morphism in the model structure for Kan complexes if and only if it is local with respect to that morphism in the model structure for quasi-categories. This can be seen as a consequence either of the fact that the model structure for Kan complexes is a Bousfield localisation of the model structure for quasi-categories (cf.\ \cite[Lemma A.4]{MR3350089}), or of the standard result that for any simplicial set $A$ and Kan complex $X$, the Kan complex $X^A$ is a model for the derived hom-space from $A$ to $X$ in the model category $\sSet_\mathrm{K}$ (see \cite[Example 17.1.4]{MR1944041}), which coincides with our model for the derived hom-space from $A$ to $X$ in the model category $\sSet_\mathrm{J}$.
\end{remark}

\begin{remark}
An alternative model for the derived hom-spaces of the model category $\sSet_\mathrm{J}$  involves the following adjunction (which we will meet again in \S\ref{secquillen}). Let $k \colon \Delta \lra \sSet$ denote the functor that sends the ordered set $[m]$ to the nerve of its groupoid reflection, i.e.\ the nerve of the contractible groupoid with the set of objects $\{0,\ldots,m\}$. 
 This functor induces an adjunction
\begin{equation} \label{joykadj}
 \xymatrix{
\sSet_\mathrm{J} \ar@<-1.5ex>[rr]^-{\hdash}_-{k^!} && \ar@<-1.5ex>[ll]_-{k_!} \sSet_\mathrm{K}
}
\end{equation}
whose left adjoint is the left Kan extension of $k \colon \Delta \lra \sSet$ along the Yoneda embedding $\Delta \lra \sSet$. By \cite[Theorem 6.22]{joyalbarcelona}, this adjunction is a Quillen adjunction between the model structures for quasi-categories and Kan complexes as indicated.  Note that, since $k([0]) = \Delta^0$, the right adjoint functor $k^!$ sends a quasi-category to a Kan complex with the same set of objects.

One can show by another application of Lemma \ref{recoghom} that for each simplicial set $A$ and quasi-category $X$, the Kan complex $k^!(X^A)$ is a model for the derived hom-space from $A$ to $X$ in the model category $\sSet_\mathrm{J}$, which is homotopy equivalent to the Kan complex $J(X^A)$ by 
\cite[Proposition 6.26]{joyalbarcelona}. For our purposes, either of these models $J(X^A)$ or $k^!(X^A)$ for the derived hom-space would suffice; but one must be chosen, and we have chosen the former. 
\end{remark}

Using Lemma \ref{qcatloclem}, we are now able to prove the following proposition.

\begin{proposition} \label{localchar}
A quasi-category is $n$-truncated if and only if it is local with respect to the boundary inclusion $\partial\Delta^{n+2} \lra \Delta^{n+2}$ in the model structure for quasi-categories. 
\end{proposition}
\begin{proof}
By Lemma \ref{qcatloclem}, it is required to prove that a quasi-category $X$ is $n$-truncated if and only if the Kan fibration 
\begin{equation} \label{localmap}
J(X^{b_{n+2}}) \colon J(X^{\Delta^{n+2}}) \lra J(X^{\partial\Delta^{n+2}})
\end{equation}
induced by the boundary inclusion $b_{n+2} \colon \partial\Delta^{n+2} \lra \Delta^{n+2}$ is a homotopy equivalence of Kan complexes, or equivalently a trivial fibration.

Let $X$ be an $n$-truncated quasi-category. 
We will prove that the morphism (\ref{localmap}) is a trivial fibration. Since $n \geq -1$, the boundary inclusion $\partial\Delta^{n+2} \lra \Delta^{n+2}$ is $0$-bijective, and so by \cite[Lemma 5.9]{joyalbarcelona} (see also \cite[Corollary 5.11]{joyalbarcelona}) the following square is a pullback square.
\begin{equation*}
\cd{
J(X^{\Delta^{n+2}}) \fatpullbackcorner \ar[d] \ar[r] & X^{\Delta^{n+2}} \ar[d] \\
J(X^{\partial \Delta^{n+2}}) \ar[r] & X^{\partial \Delta^{n+2}}
}
\end{equation*}
Hence it suffices to prove that the morphism $X^{\Delta^{n+2}} \lra X^{\partial\Delta^{n+2}}$ is a trivial fibration. By adjointness, this is so if and only if $X$ has the right lifting property with respect to the pushout-product of the boundary inclusion $\partial\Delta^m \lra \Delta^m$ with the $(n+1)$-bijective boundary inclusion $\partial\Delta^{n+2} \lra \Delta^{n+2}$ for every $m \geq 0$. But every such pushout-product is an $(n+1)$-bijective monomorphism, and so $X$ has the desired lifting property by Corollary \ref{acyclift}. Therefore the morphism (\ref{localmap}) is a trivial fibration.

Conversely, let $X$ be a quasi-category and suppose that the morphism (\ref{localmap}) is a trivial fibration. By Proposition \ref{truncisacyc}, it remains to prove that $X$ has the right lifting property with respect to the boundary inclusion $\partial\Delta^m \lra \Delta^m$ for every $m \geq n+2$. Since trivial fibrations are surjective on $0$-simplices, it suffices to prove that the morphism $J(X^{b_m}) \colon J(X^{\Delta^m}) \lra J(X^{\partial\Delta^m})$ is a trivial fibration for every $m \geq n+2$.

We prove by induction that the morphism $J(X^{b_m})$ is a trivial fibration for every $m \geq n+2$. The base case $m = n+2$ of the induction is precisely the assumption that the morphism (\ref{localmap}) is a trivial fibration. Now suppose $m > n+2$, 
and let $0 < i < m$ be an integer (which exists since $n \geq -1$). We then have a diagram of monomorphisms as on the left below,
\begin{equation*}
\cd{
\partial\Delta^{m-1} \ar[d]_-{b_{m-1}} \ar[r] & \Lambda^m_i \ar[d] \ar@/^1pc/[ddr]^-{h_m^i} \\
\Delta^{m-1} \ar@/_1pc/[drr]_-{\delta^i} \ar[r] & \partial \Delta^m \ar[dr]^-{b_m} \fatpushoutcorner \\ 
&& \Delta^m
}
\qquad
\qquad
\cd{
J(X^{\Delta^m}) \ar[dr]|-{J(X^{b_m})} \ar@/^1pc/[drr]^-{J(X^{\delta^i})} \ar@/_1pc/[ddr]_-{J(X^{h_m^i})} \\
& J(X^{\partial\Delta^m}) \fatpullbackcorner \ar[r] \ar[d] & J(X^{\Delta^{m-1}}) \ar[d]^-{J(X^{b_{m-1}})} \\
& J(X^{\Lambda^m_i}) \ar[r] & J(X^{\partial\Delta^{m-1}})
}
\end{equation*}
and hence a diagram of Kan fibrations as on the right above. In this latter diagram, the morphism $J(X^{b_{m-1}})$ is a trivial fibration by the induction hypothesis, and hence so is its pullback. Since the morphism $h_m^i$ is an inner horn inclusion, the morphism $X^{h_m^i}$ is a trivial fibration, and hence so is the morphism $J(X^{h_m^i})$. It then follows from the two-of-three property that the morphism $J(X^{b_m})$ is a trivial fibration. This completes the proof by induction.
\end{proof}

As a special case of this result, we recover the following well-known characterisation of $n$-types (cf.\ \cite[Proposition 1.5.1]{MR1944041}).

\begin{corollary} \label{kanloccor}
A Kan complex $X$ is an $n$-type if and only if it is local with respect to the boundary inclusion $\partial\Delta^{n+2} \lra \Delta^{n+2}$ in the model structure for Kan complexes.
\end{corollary}
\begin{proof}
By Remark \ref{howkanloc}, a Kan complex is local with respect to the boundary inclusion $\partial\Delta^{n+2} \lra \Delta^{n+2}$ in the model structure for Kan complexes if and only if it is local with respect to it in the model structure for quasi-categories. Hence the result follows from Proposition \ref{localchar} and Corollary \ref{kancor}.
\end{proof}

\begin{remark} \label{thelowestcase}
A Kan complex $X$ is local with respect to the boundary inclusion $\partial\Delta^0 = \emptyset \lra \Delta^0$ if and only if the unique morphism $X \cong X^{\Delta^0} \lra X^{\emptyset} = \Delta^0$ is a homotopy equivalence, that is, if and only if $X$ is contractible. Hence Corollary \ref{kanloccor} holds for all $n \geq -2$.
\end{remark}

We may now apply Smith's existence theorem (Theorem \ref{smiththm}) to deduce the existence of the Bousfield localisation of the model structure for quasi-categories whose fibrant objects are precisely the $n$-truncated quasi-categories. We break the statement of the following result into two parts:\ the first part was stated without proof in \cite[\S26.5]{joyalnotes}, and the second part is new to this paper.

\begin{theorem}[{\cite[\S26.5]{joyalnotes}}] \label{ntruncmodstr}
There exists a model structure on the category of simplicial sets whose cofibrations are the monomorphisms and whose fibrant objects are the $n$-truncated quasi-categories. This model structure is cartesian and left proper.
\end{theorem}
\begin{theorem} \label{newthm}
The model structure of \textup{Theorem \ref{ntruncmodstr}} is the Bousfield localisation of Joyal's model structure for quasi-categories with respect to the boundary inclusion $\partial\Delta^{n+2} \lra \Delta^{n+2}$, and is combinatorial.
\end{theorem}
\begin{proof}
Since the model category $\sSet_\mathrm{J}$ is left proper and combinatorial, there exists by Theorem \ref{smiththm} a Bousfield localisation of $\sSet_\mathrm{J}$ whose fibrant objects are precisely the quasi-categories that are local with respect to the single morphism $\partial\Delta^{n+2} \lra \Delta^{n+2}$. By Proposition   \ref{localchar}, these fibrant objects are precisely the $n$-truncated quasi-categories. Theorem \ref{smiththm} further implies that this model structure is left proper and combinatorial. The model structure is cartesian by Proposition  \ref{cartprop} and Corollary \ref{cartpropqcat}, since $\sSet_\mathrm{J}$ is a cartesian model category in which every object is cofibrant.
\end{proof}

\begin{remark} \label{nosmall}
In \cite[\S26.5]{joyalnotes}, the model structure of Theorem \ref{ntruncmodstr} is defined as the Bousfield localisation of the model structure $\sSet_\mathrm{J}$ for quasi-categories with respect to the (large) class of ``weak categorical $n$-equivalences'' (defined therein as the morphisms of simplicial sets satisfying the property stated in Lemma \ref{wkcatequiv} below). However, our identification of this model structure with the Bousfield localisation of $\sSet_\mathrm{J}$ with respect to the boundary inclusion $\partial\Delta^{n+2} \lra \Delta^{n+2}$, or indeed with respect to any small set of morphisms, is not contained in \cite{joyalnotes}.
\end{remark}

We will call the model structure of Theorem \ref{ntruncmodstr} the \emph{model structure for $n$-truncated quasi-categories}. Similarly, one can prove by Corollary \ref{kanloccor} and Theorem \ref{smiththm} that the $n$-types are the fibrant objects of the Bousfield localisation of the model structure for Kan complexes with respect to the boundary inclusion $\partial\Delta^{n+2} \lra \Delta^{n+2}$, as recalled in Remark \ref{ntypemodstr}. Since every $n$-type is an $n$-truncated quasi-category by Corollary \ref{kancor}, this model structure for $n$-types is also a Bousfield localisation of the model structure for $n$-truncated quasi-categories; indeed, the following proposition implies that it is the Bousfield localisation of this model structure with respect to the unique morphism $\Delta^1 \lra \Delta^0$ (cf.\ Examples \ref{example1} and \ref{example2}).

\begin{proposition} \label{kanisloc}
A quasi-category is a Kan complex if and only if it is local with respect to the unique morphism $\Delta^1 \lra \Delta^0$ in the model structure for quasi-categories.
\end{proposition}
\begin{proof}
Let $X$ be a quasi-category. By Lemma \ref{qcatloclem}, it suffices to prove that $X$ is a Kan complex if and only if the induced morphism of Kan complexes $J(X) \lra J(X^{\Delta^1})$ is a homotopy equivalence. To prove this, consider the following commutative diagram of Kan complexes. 
\begin{equation*}
\cd[@C=1em]{
& J(X) \ar[dr] \ar[dl] \\
J(X)^{\Delta^1} \ar[rr] && J(X^{\Delta^1})
}
\end{equation*}
In this diagram, the left-diagonal morphism is a homotopy equivalence, since $\Delta^1 \lra \Delta^0$ is a homotopy equivalence. Hence, by the two-of-three property, it remains to show that $X$ is a Kan complex if and only if the bottom morphism in this diagram is a homotopy equivalence. But this bottom morphism is both a monomorphism and a Kan fibration, since, by \cite[Proposition 5.3]{joyalbarcelona}, it is the image under the functor $J$ of the inclusion $X^{(\Delta^1)} \lra X^{\Delta^1}$ of the replete full sub-quasi-category of $X^{\Delta^1}$ consisting of the isomorphisms in $X$, which is both a monomorphism and an isofibration. Hence the bottom morphism is a homotopy equivalence if and only if it is surjective on objects, which is so precisely when every morphism in the quasi-category $X$ is an isomorphism, that is, precisely when $X$ is a Kan complex.
\end{proof}

 We have constructed the model structure for $n$-truncated quasi-categories as the Bousfield localisation of the model structure for quasi-categories with respect to the boundary inclusion $\partial\Delta^{n+2} \lra \Delta^{n+2}$. However, as in Remark \ref{bouslocrmk}, this model structure can also be described as the Bousfield localisation of the model structure for quasi-categories with respect to any of a variety of alternative morphisms. 
  To conclude this section, we give one such alternative morphism. This will be derived as an instance of a more general proposition, which we will prove by an application of the following standard result.

Consider a commutative diagram of simplicial sets as displayed below,
 \begin{equation*}
 \cd[@C=1em]{
 X \ar[rr]^-f \ar[dr]_-p && Y \ar[dl]^-q	 \\
 & A
 }
 \end{equation*}
 in which the morphisms $p$ and $q$ are Kan fibrations. A standard result states that 
the morphism $f$ is a weak homotopy equivalence if and only if, for each $0$-simplex $a \in A_0$, the induced morphism between fibres $f_a \colon p^{-1}(a) \lra q^{-1}(a)$ is a homotopy equivalence of Kan complexes.

Let $\Sigma \colon \sSet \lra \sSet$ denote the (two-point) suspension functor, that is, the composite of the left adjoint of the adjunction (\ref{susphomadj}) with the functor $\partial\Delta^1 \backslash \sSet \lra \sSet$ that forgets the base points. Since the adjunction (\ref{susphomadj}) is a Quillen adjunction, the suspension functor preserves monomorphisms and sends weak homotopy equivalences to weak categorical equivalences.

\begin{proposition} \label{localhoms}
Let $f \colon A \lra B$ be a morphism of simplicial sets. A quasi-category $X$ is local  with respect to the morphism $\Sigma(f) \colon \Sigma A \lra \Sigma B$ in the model structure for quasi-categories if and only if, for each pair of objects $x,y \in X$, the hom-space $\Hom_X(x,y)$ is local with respect to the morphism $f \colon A \lra B$ in the model structure for Kan complexes.
\end{proposition}
\begin{proof}
Let $f \colon A \lra B$ be a morphism of simplicial sets and let $X$ be a quasi-category. 
By Lemma \ref{qcatloclem}, $X$ is local with respect to the morphism $\Sigma(f)$ in $\sSet_\mathrm{J}$ if and only if the morphism of Kan complexes
\begin{equation} \label{kanfib1}
J(X^{\Sigma(f)}) \colon J(X^{\Sigma B}) \lra J(X^{\Sigma A})
\end{equation} 
is a homotopy equivalence. By Lemma \ref{qcatloclem} and Remark \ref{howkanloc}, for each pair of objects $x,y$ of $X$, the hom-space $\Hom_X(x,y)$ is local with respect to the morphism $f$ in $\sSet_\mathrm{K}$ if and only if the morphism of Kan complexes 
\begin{equation}\label{kanfib2}
\Hom_X(x,y)^f \colon \Hom_X(x,y)^B \lra \Hom_X(x,y)^A
\end{equation}
is a homotopy equivalence. Hence it is required to prove that the morphism (\ref{kanfib1}) is a homotopy equivalence if and only if the morphism (\ref{kanfib2}) is a homotopy equivalence for each pair of objects $x,y$ of $X$. 

From the commutative diagram of simplicial sets on the left below
\begin{equation*}
\cd[@C=1em]{
& \partial\Delta^1 \ar[dl]_-{(\bot,\top)} \ar[dr]^-{(\bot,\top)} \\
\Sigma A \ar[rr]_-{\Sigma(f)} && \Sigma B
}
\qquad
\qquad
\qquad
\cd[@C=1em]{
J(X^{\Sigma B}) \ar[rr]^-{J(X^{\Sigma(f)})} \ar[dr] && J(X^{\Sigma A}) \ar[dl] \\
 & J(X \times X)
}
\end{equation*}
we obtain the commutative diagram on the right above, in which the diagonal morphisms are Kan fibrations. By the standard result recalled above, the morphism $J(X^{\Sigma(f)})$ is a homotopy equivalence if and only if, for each pair of objects  $x,y$ of $X$, the induced morphism between the fibres over $(x,y)$ is a homotopy equivalence. Therefore, the result follows from the observation that, for each pair of objects $x,y$ of $X$, this induced morphism between the fibres is none other than the morphism (\ref{kanfib2}).
This can be seen as follows.

For each simplicial set $U$,  since the functor $V \mapsto X^V$ sends pushouts to pullbacks,  the quasi-category $X^{\Sigma U}$ is given by the pullback on the right below.
\begin{equation*}
\cd{
\Hom_X(x,y)^U \ar[r] \ar[d] \fatpullbackcorner & X^{\Sigma U} \ar[r] \ar[d] \fatpullbackcorner & (X^{\Delta^1})^U \ar[d] \\
\Delta^0 \ar[r]_-{(x,y)} & X \times X \ar[r] & (X\times X)^U
}
\end{equation*}
Since the functor $(-)^U$ preserves limits, we see by the pasting lemma for pullbacks that the fibre of the isofibration $X^{\Sigma U} \lra X \times X$ over a pair of objects $(x,y)$ is the Kan complex $\Hom_X(x,y)^U$, and hence, upon application of the limit preserving functor $J$, that this Kan complex is also the fibre of the Kan fibration $J(X^{\Sigma U}) \lra J(X \times X)$ over $(x,y)$. A further application of the pasting lemma to the diagram
\begin{equation*}
\cd[@C=3em]{
\Hom_X(x,y)^B \fatpullbackcorner \ar[d] \ar[rr]^-{\Hom_X(x,y)^f} && \Hom_X(x,y)^A \ar[r] \ar[d] \fatpullbackcorner & \Delta^0 \ar[d]^-{(x,y)} \\
J(X^{\Sigma B}) \ar[rr]_-{J(X^{\Sigma(f)})} && J(X^{\Sigma A}) \ar[r] & J(X \times X)
}
\end{equation*}
shows that the morphism (\ref{kanfib2}) is the pullback of the morphism (\ref{kanfib1}), seen as a morphism of simplicial sets over $J(X \times X)$, along the morphism $(x,y) \colon \Delta^0 \lra J(X \times X)$, as required. 
\end{proof}

By applying this proposition to the morphism $\partial\Delta^{n+1} \lra \Delta^{n+1}$, we obtain an alternative characterisation of $n$-truncated quasi-categories as local objects, and thus an alternative description of the model structure for $n$-truncated quasi-categories as a Bousfield localisation of the model structure for quasi-categories.

\begin{corollary} \label{finalcor}
A quasi-category is $n$-truncated if and only if it is local with respect to the morphism $\Sigma(\partial\Delta^{n+1} \lra \Delta^{n+1})$ in the model structure for quasi-categories. Hence the model structure for $n$-truncated quasi-categories is the Bousfield localisation of the model structure for quasi-categories with respect to the morphism  $\Sigma(\partial\Delta^{n+1} \lra \Delta^{n+1})$.
\end{corollary}
\begin{proof}
By Corollary \ref{kanloccor} (or Remark \ref{thelowestcase}, if $n=-1$), a Kan complex is an $(n-1)$-type if and only if it is local with respect to the boundary inclusion $\partial\Delta^{n+1} \lra \Delta^{n+1}$ in the model structure for Kan complexes. Hence the result follows from Proposition  \ref{localhoms}.
\end{proof}

\section{Categorical $n$-equivalences} \label{seccatnequiv}

\emph{Throughout this section, let $n \geq 0$ be an integer.}

\medskip

A morphism of simplicial sets is said to be a \emph{weak categorical $n$-equivalence} if it is a weak equivalence in the model structure for $n$-truncated quasi-categories established in Theorems \ref{ntruncmodstr} and \ref{newthm}. Since this model structure is a Bousfield localisation of the model structure for quasi-categories, the class of weak categorical $n$-equivalences enjoys the following characterisation.

\begin{lemma}[{\cite[\S26.5]{joyalnotes}}] \label{wkcatequiv}
A morphism of simplicial sets $f \colon A \lra B$ is a weak categorical $n$-equivalence if and only if the function
\begin{equation*}
(\Ho\sSet_\mathrm{J})(f,X) : (\Ho\sSet_\mathrm{J})(B,X) \lra (\Ho\sSet_\mathrm{J})(A,X)
\end{equation*}
is a bijection for each $n$-truncated quasi-category $X$.
\end{lemma}
\begin{proof}
Since the weak categorical $n$-equivalences are the weak equivalences in the model structure for $n$-truncated quasi-categories, which is a Bousfield localisation of the model structure for quasi-categories, this is an instance of Lemma \ref{locwe}. 
\end{proof}

The main goal of this section is to prove that a morphism of quasi-categories is a weak categorical $n$-equivalence if and only if it is a categorical $n$-equivalence, in the sense of the following definitions. (We reiterate that this result was stated without proof in \cite[\S26.6]{joyalnotes}.)

\begin{definition} \label{catnequiv}
If $n \geq 1$, a morphism of quasi-categories $f \colon X \lra Y$ is said to be a \emph{categorical $n$-equivalence} if it is essentially surjective on objects, and if for each pair of objects $x,y \in X$, the induced morphism of hom-spaces $f = f_{x,y} \colon \Hom_X(x,y) \lra \Hom_Y(fx,fy)$ is a homotopy $(n-1)$-equivalence.
\end{definition}

Let us first examine the lowest dimensional case of this definition.

\begin{proposition}
A morphism of quasi-categories is a categorical $1$-equivalence if and only if it is sent by the fundamental category functor $\tau_1 \colon \sSet \lra \Cat$ to an equivalence of categories.
\end{proposition}
\begin{proof}
Recall that the restriction of the fundamental category functor to the full subcategory of quasi-categories is naturally isomorphic to the homotopy category functor. Let $f \colon X\lra Y$ be a morphism of quasi-categories. By definition, $f$ is essentially surjective on objects if and only if the induced functor between homotopy categories $\ho(f) \colon \ho X \lra \ho Y$ is  essentially surjective on objects. By construction, the functor $\ho(f)$ is fully faithful if and only if $f$ is a homotopy $0$-equivalence on hom-spaces. Therefore the morphism of quasi-categories $f$ is a categorical $0$-equivalence if and only if the functor $\ho(f)$ is an equivalence of categories.
\end{proof} 

Similarly, let us make the following definition (cf.\ Proposition \ref{posetprop}). Recall that the category $\mathbf{Pos}$ of posets is a reflective subcategory of $\Cat$; the poset reflection  of a  category $A$ is the poset quotient of its preorder reflection, where the latter is the preorder whose objects are those of $A$ and in which one has $a \leq b$ if and only if the hom-set $A(a,b)$ is nonempty. Thus one obtains a composite adjunction
\begin{equation} \label{posadj}
\xymatrix{
\mathbf{Pos} \ar@<-1.5ex>[rr]^-{\hdash}_-N && \ar@<-1.5ex>[ll]_-{\tau_p} \sSet
}
\end{equation}
whose fully faithful right adjoint sends a poset to its nerve, and whose left adjoint sends a simplicial set to the poset reflection of its fundamental category.

\begin{definition}
A morphism of quasi-categories is said to be a \emph{categorical $0$-equivalence} if it is sent by the functor $\tau_p \colon \sSet \lra \mathbf{Pos}$ to an isomorphism of posets.
\end{definition}

\begin{remark} \label{cat0equiv}
Unpacking this definition, one finds that a morphism of quasi-categories $f \colon X \lra Y$ is a categorical $0$-equivalence if and only if it satisfies the following two properties:
\begin{enumerate}[(i)]
\item for each object $z\in Y$, there exists an object $x \in X$ and a pair of morphisms $Fx \lra z$ and $z \lra Fx$ in $Y$, and
\item for each pair of objects $x,y \in X$, the induced morphism $f \colon \Hom_X(x,y) \lra \Hom_Y(fx,fy)$ is a homotopy $(-1)$-equivalence.
\end{enumerate}
If $Y$ is a $0$-truncated quasi-category, then any endomorphism in $Y$ is necessarily an isomorphism, and so a morphism of quasi-categories $f\colon X \lra Y$ satisfies property (i) if and only if it is essentially surjective on objects.
\end{remark}

\begin{remark}
To prevent a proliferation of cases, we have made the global assumption $n \geq 0$ in this section. The $n =-1$ case of the problem of this section is easily dispensed with: since the model structure for $(-1)$-truncated quasi-categories coincides with the model structure for $(-1)$-types, a morphism of simplicial sets is a weak categorical $(-1)$-equivalence if and only if it is inverted by the functor $\pi_{-1} \colon \sSet \lra \{0<1\}$ that sends the empty simplicial set to $0$ and every nonempty simplicial set to $1$.
\end{remark}

Next, we establish a few useful properties of the class of categorical $n$-equivalences.

\begin{lemma} \label{coskequiv}
Let $f \colon X \lra Y$ be a morphism of $n$-truncated quasi-categories. Then the following properties are equivalent.
\begin{enumerate}[leftmargin=*, font=\normalfont, label=(\roman*)]
\item $f$ is an equivalence of quasi-categories.
\item $f$ is a weak categorical $n$-equivalence.
\item $f$ is a categorical $n$-equivalence.
\end{enumerate}
\end{lemma}
\begin{proof}
The equivalence (i) $\iff$ (ii) is a consequence of the fact that the model structure for $n$-truncated quasi-categories is a Bousfield localisation of the model structure for quasi-categories.

To prove the equivalence (i) $\iff$ (iii), recall that a morphism of quasi-categories is an equivalence if and only if it is essentially surjective on objects and a homotopy equivalence on hom-spaces, and that a morphism between $(n-1)$-types is a homotopy equivalence if and only if it is a homotopy $(n-1)$-equivalence. Since the hom-spaces of $n$-truncated quasi-categories are $(n-1)$-types, we see that a morphism of $n$-truncated quasi-categories is an equivalence if and only if it is a categorical $n$-equivalence (by Remark \ref{cat0equiv} if $n=0$). 
\end{proof}

\begin{lemma} \label{twothree}
Let $f \colon X \lra Y$ and $g \colon Y \lra Z$ be morphisms of quasi-categories. If two of the morphisms $f,g,gf$ are categorical $n$-equivalences, then so is the third.
\end{lemma}
\begin{proof}
The class of categorical $0$-equivalences was defined as the class of morphisms of quasi-categories inverted by a functor, and therefore satisfies the stated property.

Note that by the functoriality of the hom-space construction, the composite morphism $gf \colon X \lra Z$ is given on hom-spaces by the composite morphism
\begin{equation} \label{componhoms}
\cd{
\Hom_X(x,x') \ar[r]^-f & \Hom_Y(fx,fx') \ar[r]^-g & \Hom_Z(gfx,gfx').
}
\end{equation}

Suppose $n\geq 1$.
We must consider three cases. In the first, suppose $f$ and $g$ are categorical $n$-equivalences. Since the class of essentially surjective on objects morphisms of quasi-categories and the class of homotopy $(n-1)$-equivalences of Kan complexes are both closed under composition (by Lemma \ref{homo23}), we have that the composite morphism $gf \colon X \lra Z$ is a categorical $n$-equivalence. 

In the second case, suppose that $g$ and $gf$ are categorical $n$-equivalences. To show that $f$ is essentially surjective on objects, it suffices to show that the functor $\ho(f) \colon \ho X \lra \ho Y$ is essentially surjective on objects. This follows from the assumptions  (which hold since $n \geq 1$) that the functor $\ho(gf)$ is essentially surjective on objects and that the functor $\ho(g)$ is fully faithful. Since $gf$ is given on hom-spaces by the composite (\ref{componhoms}), we have that $f$ is a homotopy $(n-1)$-equivalence on hom-spaces by Lemma \ref{homo23}.

In the third case, suppose that $f$ and $gf$ are categorical $n$-equivalences. Since $gf$ is essentially  surjective on objects, it follows that $g$ is essentially surjective on objects. To show that $g$ is a homotopy $(n-1)$-equivalence on hom-spaces, let $y,y'$ be a pair of objects of $Y$. Since $f$ is essentially surjective on objects, there exist objects $x,x' \in X$ and isomorphisms $u \colon fx \cong y$ and $v \colon fx' \cong y'$ in $Y$. Thus we have a commutative diagram of quasi-categories as on the left below in which the vertical morphisms are equivalences of quasi-categories (by the Quillen adjunction (\ref{joyalresadj})),
\begin{equation*}
\cd[@=3em]{
& Y \ar[r]^-{g} & Z \\
\partial\Delta^1 \ar[ur]^-{(y,y')} \ar[r]^-{(u,v)} \ar[dr]_-{(fx,fx')} & Y^{(\Delta^1)} \ar[u]_-{Y^{({\delta^0})}} \ar[d]^-{Y^{({\delta^1})}} \ar[r]^-{g^{(\Delta^1)}} & Z^{(\Delta^1)} \ar[u]_-{Z^{({\delta^0})}} \ar[d]^-{Z^{({\delta^1})}} \\
& Y \ar[r]_-{g} & Z
}
\qquad
\qquad
\cd[@=3em]{
\Hom_Y(y,y') \ar[r]^-{g} & \Hom_Z(gy,gy') \\
 \Hom_{Y^{(\Delta^1)}}(u,v) \ar[u]^-{Y^{({\delta^0})}} \ar[d]_-{Y^{({\delta^1})}} \ar[r]^-{g^{(\Delta^1)}} & \Hom_{Z^{(\Delta^1)}}(gu,gv) \ar[u]_-{Z^{({\delta^0})}} \ar[d]^-{Z^{({\delta^1})}} \\
\Hom_Y(fx,fx') \ar[r]_-{g} & \Hom_Z(gfx,gfx')
}
\end{equation*}
 and which therefore induces a commutative diagram of Kan complexes as on the right above in which the vertical morphisms are homotopy equivalences, and hence also homotopy $(n-1)$-equivalences. Hence, by Lemma \ref{homo23}, the morphism $g \colon \Hom_Y(y,y') \lra \Hom_Z(gy,gy')$ is a homotopy $(n-1)$-equivalence if and only if the morphism $g \colon \Hom_Y(fx,fx') \lra \Hom_Z(gfx,gfx')$ is a homotopy $(n-1)$-equivalence. But the latter morphism is a homotopy $(n-1)$-equivalence by Lemma \ref{homo23}, since the composite morphism (\ref{componhoms}) and its first factor are homotopy $(n-1)$-equivalences by assumption.
\end{proof}

By construction (\ref{susppo}), the suspension $\Sigma U$ of an $n$-skeletal simplicial set $U$ is $(n+1)$-skeletal (since it is a colimit of $(n+1)$-skeletal simplicial sets). Hence the $n$-skeleta of the hom-spaces $\Hom_X(x,y)$ of a quasi-category $X$ depend only on the $(n+1)$-skeleton of $X$. This implies that an $(n+1)$-bijective morphism of quasi-categories $f \colon X \lra Y$ induces $n$-bijective morphisms on hom-spaces $f \colon \Hom_X(x,y) \lra \Hom_Y(fx,fy)$. We may therefore deduce the following lemma from Lemma \ref{bijimpho}.

\begin{lemma} \label{bijequiv}
An $(n+1)$-bijective morphism of quasi-categories is a categorical $n$-equivalence.
\end{lemma}
\begin{proof}
Let $f \colon X \lra Y$ be an $(n+1)$-bijective morphism of quasi-categories. Then $f$ is a $0$-bijection, and hence in particular (essentially) surjective on objects (if $n=0$, note that this implies property (i) of Remark \ref{cat0equiv}). Furthermore, for each pair of objects $x,y$ of $X$, the induced morphism on hom-spaces $\Hom_X(x,y) \lra \Hom_{Y}(fx,fy)$ is $n$-bijective as above, and hence is a homotopy $(n-1)$-equivalence by Lemma \ref{bijimpho}. Therefore $f$ is a categorical $n$-equivalence.
\end{proof}

Following \cite[\S26.7]{joyalnotes}, define a \emph{categorical $n$-truncation} of a simplicial set $A$ to be a fibrant replacement of $A$ in the model structure for $n$-truncated quasi-categories, that is, an $n$-truncated quasi-category $X$ together with a weak categorical $n$-equivalence $A \lra X$. In the next two propositions, we will prove that the $(n+1)$-coskeleton of a quasi-category is a model for its categorical $n$-truncation (cf.\ \cite[\S1]{MR0245577} or \cite[\S9.1]{MR2294028}, where the $(n+1)$-coskeleton of a Kan complex is given as a model for its $n$th Postnikov truncation). We will then use these results to prove the main theorem of this section.

\begin{proposition} \label{coskqcat}
Let $X$ be a quasi-category. Then its $(n+1)$-coskeleton $\cosk_{n+1}X$ is an $n$-truncated quasi-category, and the unit morphism $X \lra \cosk_{n+1}X$ is a categorical $n$-equivalence.
\end{proposition}
\begin{proof}
First, to prove that $\cosk_{n+1}X$ is a quasi-category, it is required to prove that it has the right lifting property with respect to the inner horn inclusion $h_m^k \colon \Lambda^m_k \lra \Delta^m$ for every $m \geq 2$ and $0 < k < m$. By adjointness, this is so if and only if $X$ has the right lifting property with respect to the morphism $\sk_{n+1}(h_m^k) \colon \sk_{n+1}\Lambda^m_k \lra \sk_{n+1}\Delta^m$. 
Consider the following three cases. If $m \leq n+1$, then the morphism $\sk_{n+1}(h_m^k)$ is the inner horn inclusion $h_m^k$, with respect to which $X$ has the right lifting property since it is a quasi-category. If $m = n+2$, then the morphism $\sk_{n+1}(h_m^k)$ is isomorphic to the inclusion $\Lambda^m_k \lra \partial\Delta^m$, with respect to which $X$ has the right lifting property, since it has this property with respect to the composite $\Lambda^m_k \lra \partial\Delta^m \lra \Delta^m$, since it is a quasi-category. If $m > n+2$, then the morphism $\sk_{n+1}(h^m_k)$ is an isomorphism, with respect to which therefore $X$ has the unique right lifting property. 

Next, to show that the quasi-category $\cosk_{n+1}X$ is $n$-truncated, it suffices to observe that the unit morphism $\cosk_{n+1}X \lra \cosk_{n+1}\cosk_{n+1}X$ is an isomorphism (since $\cosk_{n+1}$ is an idempotent monad), for then $\cosk_{n+1}$ is $n$-truncated by Corollary \ref{qunittrivfib}.

Finally, since the unit morphism $X \lra \cosk_{n+1}X$ is an $(n+1)$-bijective morphism of quasi-categories, it is a categorical $n$-equivalence by Lemma \ref{bijequiv}. 
 \end{proof}

 Let $J = k([1])$ denote the nerve of the ``free-living isomorphism'', i.e.\ the nerve of the groupoid reflection of the ordered set $\{ 0< 1\}$. 
By \cite[Proposition 6.18]{joyalbarcelona},
 for any simplicial set $A$ and quasi-category $X$, the hom-set $(\Ho \sSet_\mathrm{J})(A,X)$ is in bijection with the set of $J$-homotopy classes of morphisms $A \lra X$, where two such morphisms $f,g$ belong to the same $J$-homotopy class if and only if there exists a morphism $h \colon J \times A \lra X$ such that $h \circ (\{0\} \times \mathrm{id}) = f$ and $h \circ (\{1\}\times\mathrm{id}) = g$.

\begin{proposition} \label{unitwe}
Let $A$ be a simplicial set. Then the unit morphism $A \lra \cosk_{n+1}A$ is a weak categorical $n$-equivalence. 
\end{proposition}
\begin{proof}
Let $\eta_A \colon A \lra \cosk_{n+1}A$ denote the unit morphism in question. By Lemma \ref{wkcatequiv}, it is required to prove that the function 
\begin{equation*} \label{hofun}
(\Ho\sSet_\mathrm{J})(\eta_A,X) : (\Ho\sSet_\mathrm{J})(\cosk_{n+1}A,X) \lra (\Ho\sSet_\mathrm{J})(A,X)
\end{equation*}
is a bijection for each $n$-truncated quasi-category $X$, which, without loss of generality,  we may assume to be $(n+1)$-coskeletal by Lemma \ref{qunittrivfib}.

Let $X$ be an $(n+1)$-coskeletal quasi-category. To show that the function displayed above is injective, let $f,g \colon \cosk_{n+1}A \lra X$ be a pair of morphisms of simplicial sets, and let $h \colon J\times A \lra X$ be a $J$-homotopy from $f\eta_A$ to $g\eta_A$. Then the morphism
\begin{equation*}
\cd[@C=4em]{
J \times \cosk_{n+1}A \cong  \cosk_{n+1}(J \times A) \ar[r]^-{\cosk_{n+1}(h)} & \cosk_{n+1}X \cong X
}
\end{equation*}
defines a $J$-homotopy from $f$ to $g$ (where we have used that the functor $\cosk_{n+1}$ preserves products and that $J$ is $0$-coskeletal). Hence the function is injective. To show that it is surjective, let $f \colon X \lra Y$ be a morphism of simplicial sets. Then the morphism 
\begin{equation*}
\cd[@C=4em]{
\cosk_{n+1}X \ar[r]^-{\cosk_{n+1}(f)} & \cosk_{n+1}Y \cong Y
}
\end{equation*}
defines an extension of $f$ along the unit morphism $\eta_A$. Hence the function is surjective, and is therefore a bijection.
\end{proof}

We are now ready to prove the main theorem of this section.

\begin{theorem}[{\cite[\S26.6]{joyalnotes}}] \label{wethm}
A morphism of quasi-categories is a weak categorical $n$-equivalence if and only if it is a categorical $n$-equivalence.
\end{theorem}
\begin{proof}
This statement is true of morphisms of $n$-truncated quasi-categories by Lemma \ref{coskequiv}. 
Let $f \colon X \lra Y$ be a morphism of quasi-categories. In the commutative diagram displayed below,
\begin{equation*}
\cd{
X \ar[rr]^-{f} \ar[d] && Y \ar[d] \\
\cosk_{n+1}X \ar[rr]_-{\cosk_{n+1}(f)} && \cosk_{n+1}Y
}
\end{equation*}
the vertical morphisms are weak categorical $n$-equivalences by Proposition \ref{unitwe} and categorical $n$-equivalences by Proposition \ref{coskqcat}, and the bottom morphism is a morphism of $n$-truncated quasi-categories by Proposition \ref{coskqcat}. Since the class of weak categorical $n$-equivalences and the class of categorical $n$-equivalences both satisfy the two-of-three property (the one since it is the class of weak equivalences of a model category by definition, the other by Lemma \ref{twothree}), it follows that $f$ inherits from $\cosk_{n+1}(f)$ the property that it is a weak categorical $n$-equivalence if and only if it is a categorical $n$-equivalence.
\end{proof}

\begin{remark}
In \cite[\S26.6]{joyalnotes}, it is incorrectly stated that a morphism of quasi-categories is a (weak) categorical $0$-equivalence if and only if it is essentially surjective on objects and a homotopy $(-1)$-equivalence on hom-spaces. This statement can be corrected by replacing the property ``essentially surjective on objects'' by the weaker property (i) in Remark \ref{cat0equiv}.  For a counterexample, let $C$ be the category freely generated by the graph displayed below,
\begin{equation*}
\cd{
\bullet \ar@<1ex>[r] & \bullet \ar@<1ex>[l]
}
\end{equation*}
and let $1 \lra C$ be the functor corresponding to either of the two objects of $C$. This functor is not essentially surjective on objects, but its poset reflection is an isomorphism. Hence the nerve of this functor is an example of a categorical $0$-equivalence that is not essentially surjective on objects.
\end{remark}

\section{Some Quillen equivalences} \label{secquillen}
In this final section, we use the criteria proved at the end of Appendix \ref{secbousfield} to prove Quillen equivalences between the model categories of categories and $1$-truncated quasi-categories and between the model categories of $n$-truncated quasi-categories and Rezk's $(n,1)$-$\Theta$-spaces. 

To begin, recall that the adjunction $\tau_1 \dashv N \colon \Cat \lra \sSet$ (\ref{nerveadj}),
whose right adjoint sends a category $A$ to its nerve $NA$ and whose left adjoint sends a simplicial set $X$ to its fundamental category $\tau_1X$, is a Quillen adjunction, and moreover a homotopy reflection (i.e.\ its derived right adjoint is fully faithful), between the folk model structure for categories (whose weak equivalences are the equivalences of categories) and Joyal's model structure for quasi-categories \cite[Proposition 6.14]{joyalbarcelona}. Using Theorem \ref{qethm1} and the results of \S\ref{sectruncated}, we can show that this adjunction is moreover a Quillen equivalence between the folk model structure for categories  and the model structure for $1$-truncated quasi-categories.

\begin{theorem} \label{cat1qcat}
The adjunction
\begin{equation*}
\xymatrix{
\Cat \ar@<-1.5ex>[rr]^-{\hdash}_-{N} && \ar@<-1.5ex>[ll]_-{\tau_1} \sSet
}
\end{equation*}
 is a Quillen equivalence between the folk model structure for categories and the model structure for $1$-truncated quasi-categories. 
\end{theorem}
\begin{proof}
By Theorem \ref{qethm1}, we must prove that the nerve of a category is a $1$-truncated quasi-category, and that, for any $1$-truncated quasi-category $X$, the unit morphism $X \lra N(\tau_1 X)$ is an equivalence of quasi-categories. These both follow from Proposition \ref{catprop}. 
\end{proof}

\begin{corollary}[{\cite[\S26.6]{joyalnotes}}] \label{wk1qcateq}
A morphism of simplicial sets is a weak categorical $1$-equivalence if and only if it sent by the functor $\tau_1 \colon \sSet \lra \Cat$ to an equivalence of categories.
\end{corollary}
\begin{proof}
Since the functor $\tau_1$ is the left adjoint of a Quillen equivalence by Theorem \ref{cat1qcat}, and since every simplicial set is cofibrant in the model structure for $1$-truncated quasi-categories, this is an instance of the fact that the left adjoint of a Quillen equivalence preserves and reflects weak equivalences between cofibrant objects.
\end{proof}

Recall the adjunction $\tau_p \dashv N \colon \mathbf{Pos} \lra \sSet$ (\ref{posadj}), whose fully faithful right adjoint sends a poset to its nerve, and whose left adjoint sends a simplicial set to the poset reflection of its fundamental category. We now show that this adjunction is a Quillen equivalence between the trivial model structure  (i.e.\ the unique model structure whose weak equivalences are the isomorphisms) on the category of posets and the model structure for $0$-truncated quasi-categories.

\begin{theorem} \label{cat0qcat}
The adjunction
\begin{equation*}
\xymatrix{
\mathbf{Pos} \ar@<-1.5ex>[rr]^-{\hdash}_-{N} && \ar@<-1.5ex>[ll]_-{\tau_p} \sSet
}
\end{equation*}
is a Quillen equivalence between the trivial model structure for posets and the model structure for $0$-truncated quasi-categories. 
\end{theorem}
\begin{proof}
To see that this adjunction is a Quillen adjunction between the trivial model structure for posets and the model structure for quasi-categories, it suffices to observe that each weak categorical equivalence is sent by the functor $\tau_p$ to an isomorphism of posets. But this functor is the composite of the functor $\tau_1 \colon \sSet \lra \Cat$, which sends each weak categorical equivalence to an equivalence of categories, and the poset reflection functor $\tau_p \colon \Cat \lra \mathbf{Pos}$, which is easily shown to invert equivalences of categories. 

It remains to verify conditions (i) and (ii) of Theorem \ref{qethm1}. Firstly, by Proposition \ref{posetprop}, the nerve $NA$ of a poset $A$ is a $0$-truncated quasi-category, which verifies condition (i). Secondly, a $0$-truncated quasi-category $X$ is in particular $1$-truncated, and so by Proposition \ref{catprop} the unit morphism $X \lra N(\ho X)$ is an equivalence of quasi-categories. But by Proposition \ref{posetprop}, $\ho X$ is a preorder and hence $N(\ho X)$ is a $0$-truncated quasi-category. This verifies condition (ii). 
\end{proof}

\begin{corollary}[{\cite[\S26.6]{joyalnotes}}] \label{wk0qcateq}
A morphism of simplicial sets is a weak categorical $0$-equivalence if and only if it is sent by the functor $\tau_p \colon \sSet \lra \mathbf{Pos}$ to an isomorphism of posets.
\end{corollary}
\begin{proof}
Since the functor $\tau_p$ is the left adjoint of a Quillen equivalence by Theorem \ref{cat0qcat}, and since every simplicial set is cofibrant in the model structure for $0$-truncated quasi-categories, this is another instance of the fact that the left adjoint of a Quillen equivalence preserves and reflects weak equivalences between cofibrant objects.
\end{proof}

Recall that a categorical $n$-truncation of a simplicial set $A$ is an $n$-truncated quasi-category $X$ together with a weak categorical $n$-equivalence $A \lra X$.

\begin{corollary}[{\cite[\S26.7]{joyalnotes}}]
For each simplicial set $A$, the unit morphism $A \lra N(\tau_1A)$ is a categorical $1$-truncation of $A$, and the unit morphism $A \lra N(\tau_pA)$ is a categorical $0$-truncation of $A$.
\end{corollary}
\begin{proof}
By Corollaries \ref{wk1qcateq} and \ref{wk0qcateq}, it suffices to show that these unit morphisms are sent to isomorphisms by the functors $\tau_1$ and $\tau_p$ respectively. In each case, this is an instance of the fact that each component of the unit of an adjunction whose right adjoint is fully faithful is sent by the left adjoint to an isomorphism.
\end{proof}

Next, we prove, for each $n \geq -1$, two Quillen equivalences between the model categories of $n$-truncated quasi-categories and Rezk's $(n,1)$-$\Theta$-spaces. 
 In \cite{MR2342834}, Joyal and Tierney established two Quillen equivalences
\begin{equation} \label{jtqe}
\xymatrix{
[\Delta^\mathrm{op},\Set] \ar@<-1.5ex>[rr]^-{\hdash}_-{t^!} && \ar@<-1.5ex>[ll]_-{t_!} [(\Delta\times\Delta)^\mathrm{op},\Set]
}
\qquad\quad
\xymatrix{
[(\Delta\times\Delta)^\mathrm{op},\Set] \ar@<-1.5ex>[rr]^-{\hdash}_-{i_1^*} && \ar@<-1.5ex>[ll]_-{p_1^*} [\Delta^\mathrm{op},\Set]
}
\end{equation}
between Joyal's model structure for quasi-categories and Rezk's model structure for complete Segal spaces on the category of bisimplicial sets (defined in \cite{MR1804411}). Suffice it to recall  that the functor $t^!$ sends a simplicial set $A$ to the bisimplicial set $t^!(A)$ whose $n$th column $t^!(A)_{n}$ is the simplicial set $k^!(A^{\Delta^n})$ (where $k^!$ denotes the right adjoint of the Quillen adjunction (\ref{joykadj})), that the functor $i_1^*$ sends a bisimplicial set $X$ to its zeroth row $X_{*0}$, and that there are natural isomorphisms $t_!p_1^* \cong \mathrm{id}$ and $i_1^*t^! \cong \mathrm{id}$. For each  complete Segal space $X$, we refer to the elements of the set $X_{00}$ as the objects of $X$; there is an evident bijection between the objects of a quasi-category $A$ and the objects of its associated complete Segal space $t^!(A)$.

 For each $n \geq -1$, Rezk constructed in \cite{MR2578310} a Bousfield localisation of the model structure for complete Segal spaces, whose fibrant objects are the complete Segal spaces $X$ each of whose hom-spaces $\Hom_X(x,y)$ is an $(n-1)$-type \cite[Proposition 11.20]{MR2578310}.  Rezk calls complete Segal spaces with this property $(n,1)$-$\Theta$-spaces, but for convenience we will call them \emph{$n$-truncated complete Segal spaces}, and we will call this model structure the \emph{model structure for $n$-truncated complete Segal spaces}. Recall from \cite[\S5.1]{MR1804411} that for each pair of objects $x,y$ of a complete Segal space $X$, the hom-space $\Hom_X(x,y)$ is defined to be the pullback
 \begin{equation*}
 \cd{
 \Hom_X(x,y) \ar[r] \ar[d] \fatpullbackcorner & X_{1} \ar[d]^-{(d_1,d_0)} \\
 \Delta^0 \ar[r]_-{(x,y)} & X_{0} \times X_{0}
 }
 \end{equation*}
 in the category of simplicial sets (where $X_{n}$ denotes the $n$th column of the bisimplicial set $X$). By comparison with the definition of the hom-spaces of a quasi-category (\ref{hompb}), one sees that there is a canonical isomorphism 
 \begin{equation} \label{homiso}
 \Hom_{t^!(A)}(x,y) \cong k^!(\Hom_A(x,y))
 \end{equation}
  for each pair of objects $x,y$ in a quasi-category $A$, since the right adjoint functor $k^!$ preserves limits.

We now apply Theorem \ref{qethm2} to prove that the two Quillen equivalences (\ref{jtqe}) remain Quillen equivalences between the Bousfield localisations for $n$-truncated quasi-categories and $n$-truncated complete Segal spaces. The following proposition shows that these Quillen equivalences satisfy the hypotheses of that theorem.

\begin{proposition} \label{hypo} Let $n \geq -1$ be an integer.
\begin{enumerate}[leftmargin=*, font=\normalfont]
\item A quasi-category $A$ is $n$-truncated if and only if the complete Segal space $t^!(A)$ is $n$-truncated. 
\item A complete Segal space $X$ is $n$-truncated if and only if its underlying quasi-category $i_1^*(X)$ is $n$-truncated.
\end{enumerate}
\end{proposition}
\begin{proof}
(1) Let $A$ be a quasi-category. For each pair of objects $x,y \in A$, there is a homotopy equivalence $\Hom_{t^!(A)}(x,y) \simeq \Hom_A(x,y)$ by the isomorphism (\ref{homiso}) and \cite[Proposition 6.26]{joyalbarcelona}. Hence the hom-spaces of $A$ are $(n-1)$-types if and only if the hom-spaces of $t^!(A)$ are $(n-1)$-types, that is, $A$ is an $n$-truncated quasi-category if and only if $t^!(A)$ is an $n$-truncated complete Segal space.

(2) Let $X$ be a complete Segal space. There is a span of weak equivalences in the model structure for complete Segal spaces
\begin{equation*}
\cd{
X & \ar[l] p_1^*(i_1^*X) \ar[r] & t^!(i_1^*X),
}
\end{equation*}
where the left-pointing arrow is the counit of the Quillen equivalence $p_1^* \dashv i_1^*$ and the right-pointing arrow is the transpose of the canonical isomorphism $t_!(p_1^*(i_1^*X)) \cong i_1^*X$ under the Quillen equivalence $t_! \dashv t^!$, both of which are weak equivalences since $X$ is fibrant. Hence $X$ is weakly equivalent to the complete Segal space $t^!(i_1^*X)$, and so $X$ is $n$-truncated if and only if $t^!(i_1^*X)$ is $n$-truncated, which by (1) is so if and only if the quasi-category $i_1^*X$ is $n$-truncated.
\end{proof}

Hence the adjunctions (\ref{jtqe}) satisfy the hypotheses of Theorem \ref{qethm2}, and we may deduce the following theorem.

\begin{theorem} \label{jtrezkthm} For each integer $n \geq -1$, the adjunctions
\begin{equation*}
\xymatrix{
[\Delta^\mathrm{op},\Set] \ar@<-1.5ex>[rr]^-{\hdash}_-{t^!} && \ar@<-1.5ex>[ll]_-{t_!} [(\Delta\times\Delta)^\mathrm{op},\Set]
}
\qquad\quad
\xymatrix{
[(\Delta\times\Delta)^\mathrm{op},\Set] \ar@<-1.5ex>[rr]^-{\hdash}_-{i_1^*} && \ar@<-1.5ex>[ll]_-{p_1^*} [\Delta^\mathrm{op},\Set]
}
\end{equation*}
are Quillen equivalences between the model structure for $n$-truncated quasi-categories on the category of simplicial sets and the model structure for $n$-truncated complete Segal spaces on the category of bisimplicial sets.
\end{theorem}
\begin{proof}
By \cite{MR2342834}, these adjunctions are Quillen equivalences between the model structures for quasi-categories and complete Segal spaces. For each $n \geq -1$, the model structures for $n$-truncated quasi-categories and $n$-truncated complete Segal spaces are Bousfield localisations of the former model structures, and so it remains to show that these adjunctions satisfy the conditions of Theorem \ref{qethm2}. But this is precisely what was shown in Proposition \ref{hypo}.
\end{proof}

\begin{remark}
In \cite[\S11]{MR2578310}, Rezk defines the model structure for $n$-truncated complete Segal spaces as the Bousfield localisation of the model structure for complete Segal spaces with respect to the morphism denoted therein by $V[1](\partial\Delta^{n+1} \lra \Delta^{n+1})$. One can show that the left adjoint functor $t_!$ sends this morphism to the morphism of simplicial sets $\Sigma(k_!(\partial\Delta^{n+1} \lra \Delta^{n+1}))$. Since there is a natural weak homotopy equivalence $\mathrm{id} \lra k_!$ \cite[Theorem 6.22]{joyalbarcelona}, and since the suspension functor $\Sigma$ sends weak homotopy equivalences to weak categorical equivalences, it follows from Corollary \ref{finalcor} that the model structure for $n$-truncated quasi-categories is the Bousfield localisation of the model structure for quasi-categories with respect to the morphism $\Sigma(k_!(\partial\Delta^{n+1} \lra \Delta^{n+1}))$. Hence one could alternatively prove Proposition \ref{hypo}(1) and that half of Theorem \ref{jtrezkthm} concerning the adjunction $t_! \dashv t^!$ by \cite[Proposition 3.1.12]{MR1944041} and \cite[Theorem 3.3.20]{MR1944041} respectively.
\end{remark}

To conclude, we combine two previous theorems to deduce a Quillen equivalence between the folk model structure for categories and the model structure for $1$-truncated complete Segal spaces. The right adjoint of this Quillen equivalence is the composite functor $t^!N \colon \Cat \lra [(\Delta\times\Delta)^\mathrm{op},\Set]$, which sends a category $A$ to its \emph{classifying diagram} \cite[\S3.5]{MR1804411}, which is the bisimplicial set whose $n$th column is the nerve of the maximal subgroupoid of the category $A^{[n]}$, and which was shown directly by Rezk to be a complete Segal space \cite[Proposition 6.1]{MR1804411}.

\begin{theorem} \label{thmqrezk}
The composite adjunction
\begin{equation*} 
\xymatrix{
\Cat \ar@<-1.5ex>[rr]^-{\hdash}_-N && \ar@<-1.5ex>[ll]_-{\tau_1} [\Delta^\mathrm{op},\Set] \ar@<-1.5ex>[rr]^-{\hdash}_-{t^!} && \ar@<-1.5ex>[ll]_-{t_!} [(\Delta\times\Delta)^\mathrm{op},\Set],
}
\end{equation*}
whose right adjoint is Rezk's classifying diagram functor, is a Quillen equivalence between the folk model structure for categories and the model structure for $1$-truncated complete Segal spaces.
\end{theorem}
\begin{proof}
This adjunction is the composite of the Quillen equivalence of Theorem \ref{cat1qcat} and the $n=1$ case of one of the Quillen equivalences of Theorem \ref{jtrezkthm}, and is therefore a Quillen equivalence.
\end{proof}

\appendix

\section{Bousfield localisations} \label{secbousfield}
In this appendix, we recall some of the basic theory of Bousfield localisations of model categories (mostly those results that we use which are difficult to find explicitly stated in the literature in the form we use them), including two criteria for detecting Quillen equivalences between Bousfield localisations. We assume familiarity with the basic theory of model categories, such as is contained in \cite[Chapter 1]{MR1650134}; our approach is particularly influenced by the insightful appendices \cite[\S7]{MR2342834} and \cite[Appendix E]{joyalbarcelona}.

We begin with the notion of a Bousfield localisation of a model category (after \cite[Definition 7.20]{MR2342834}, in contrast to \cite[Definition 3.3.1]{MR1944041}, where a Bousfield localisation is defined with respect to a given class of morphisms). Recall that a \emph{model category} is a locally small complete and cocomplete category equipped with a \emph{model structure}, which is determined by its classes $(\mathcal{C},\mathcal{W},\mathcal{F})$ of \emph{cofibrations}, \emph{weak equivalences}, and \emph{fibrations}.

\begin{definition} \label{bousdef}
A \emph{Bousfield localisation} of a model structure $(\mathcal{C},\mathcal{W},\mathcal{F})$ on a category $\mathcal{M}$ is a model structure $(\mathcal{C}_\mathrm{loc},\mathcal{W}_\mathrm{loc},\mathcal{F}_\mathrm{loc})$ on the same category $\mathcal{M}$ such that $\mathcal{C}_\mathrm{loc} = \mathcal{C}$ and $\mathcal{W} \subseteq \mathcal{W}_\mathrm{loc}$.
\end{definition}

We will often denote the model category determined by a Bousfield localisation of (the model structure of) a model category $\mathcal{M}$ by $\mathcal{M}_\mathrm{loc}$, and call the morphisms belonging to the classes $\mathcal{W}_\mathrm{loc}$ and $\mathcal{F}_\mathrm{loc}$ \emph{local weak equivalences} and \emph{local fibrations} respectively; the fibrant objects of the model category $\mathcal{M}_\mathrm{loc}$ we will call \emph{local fibrant objects}.  
It is immediate from the definition that the adjunction
\begin{equation*} \label{locqadj}
\xymatrix{
\mathcal{M}_\mathrm{loc}\ar@<-1.5ex>[rr]^-{\hdash}_-{1_\mathcal{M}} && \ar@<-1.5ex>[ll]_-{1_\mathcal{M}} \mathcal{M},
}
\end{equation*}
whose left and right adjoints both are  the identity functor on (the underlying category of) $\mathcal{M}$, is a Quillen adjunction. Hence every local fibration and local fibrant object is in particular a fibration and a fibrant object (in the model category $\mathcal{M}$) respectively. Moreover, the derived right adjoint of this Quillen adjunction is fully faithful, which is to say that the Quillen adjunction is a \emph{homotopy reflection} in the sense of \cite[Definition E.2.15]{joyalbarcelona} (the term \emph{homotopy localisation} is used in \cite{MR2342834}); it follows that a morphism between local fibrant objects is a weak equivalence if and only if it is a local weak equivalence \cite[Proposition 7.18]{MR2342834}. Furthermore, it follows by a factorisation and retract argument that a morphism between local fibrant objects is a fibration if and only if it is a local fibration \cite[Proposition 7.21]{MR2342834}.

The model category axioms imply that a morphism is a local fibration if and only if it has the right lifting property with respect to the class of morphisms $\mathcal{C}\cap\mathcal{W}_\mathrm{loc}$, whose members we call \emph{local trivial cofibrations}. 
Hence a Bousfield localisation of a model category is determined by its class $\mathcal{W}_\mathrm{loc}$ of local weak equivalences. Alternatively, a Bousfield localisation of a model category is determined by its class of local fibrant objects, since this class determines the local weak equivalences by the following argument (cf.\ \cite[Proposition E.1.10]{joyalbarcelona} and \cite[\S3.5]{MR1944041}). We denote the homotopy category of a model category $\mathcal{M}$ by $\Ho\mathcal{M}$; we will typically not distinguish an object or morphism of $\mathcal{M}$ from its image under the localisation functor $\mathcal{M} \lra \Ho\mathcal{M}$.

\begin{lemma} \label{locwe}
Let $\mathcal{M}_\mathrm{loc}$ be a Bousfield localisation of a model category $\mathcal{M}$. A morphism $f \colon A \lra B$ in $\mathcal{M}$ is a local weak equivalence if and only if the function
\begin{equation} \label{precomp}
(\Ho\mathcal{M})(f,X) : (\Ho\mathcal{M})(B,X) \lra (\Ho\mathcal{M})(A,X)
\end{equation}
is a bijection for each local fibrant object $X$.
\end{lemma}
\begin{proof}
A morphism $f \colon A \lra B$ in $\mathcal{M}$ is a local weak equivalence if and only if it is (sent to) an isomorphism in the homotopy category $\Ho\mathcal{M}_\mathrm{loc}$, which is so, by the Yoneda lemma, if and only if the  function
\begin{equation} \label{precomp2}
(\Ho\mathcal{M}_\mathrm{loc})(f,X) : (\Ho\mathcal{M}_\mathrm{loc})(B,X) \lra (\Ho\mathcal{M}_\mathrm{loc})(A,X)
\end{equation}
 is a bijection for each local fibrant object $X$ (since every object of $\Ho\mathcal{M}_\mathrm{loc}$ is isomorphic to a local fibrant object). 
 By taking cofibrant replacements in $\mathcal{M}$, we may  suppose $f \colon A \lra B$ to be a morphism between cofibrant objects.
For each cofibrant object $C$ and local fibrant object $X$, the sets $(\Ho\mathcal{M}_\mathrm{loc})(C,X)$ and $(\Ho\mathcal{M})(C,X)$ are in bijection with the sets of homotopy classes of morphisms $C \lra X$ in the model categories $\mathcal{M}_\mathrm{loc}$ and $\mathcal{M}$ respectively. But these latter sets coincide, since any cylinder object for $C$ in the model category $\mathcal{M}$ is also a cylinder object for $C$ in the model category $\mathcal{M}_{\mathrm{loc}}$; hence there is a bijection $(\Ho\mathcal{M}_\mathrm{loc})(C,X) \cong (\Ho\mathcal{M})(C,X)$. The functions (\ref{precomp}) and (\ref{precomp2}) correspond under these bijections, and so one is a bijection if and only if the other is.
\end{proof}

Hence a Bousfield localisation of a model category $\mathcal{M}$ can equivalently be defined as a model structure with the same underlying category and the same class of cofibrations as $\mathcal{M}$, but with fewer fibrant objects. This alternative definition makes it easy to recognise Bousfield localisations, as in the following example. 

\begin{example} \label{example1}
On the category $\sSet$ of simplicial sets, the model structure for Kan complexes is a Bousfield localisation of the model structure for quasi-categories, since the cofibrations are the monomorphisms in both model structures, and since every Kan complex is a quasi-category (see \cite[\S1]{MR2342834} for details). Hence a morphism of Kan complexes is a homotopy equivalence if and only if it is an equivalence of quasi-categories. 
\end{example}

We may thus regard a Bousfield localisation of a given model category as determined by its local fibrant objects. Given this perspective, the following lemma will be found useful (cf.\ \cite[Proposition E.2.23]{joyalbarcelona} and \cite[Proposition 3.3.15]{MR1944041}). 
We say that two objects in a model category $\mathcal{M}$ are \emph{weakly equivalent} if they are isomorphic in the homotopy category $\Ho\mathcal{M}$, that is, if they are connected by a zig-zag of weak equivalences in $\mathcal{M}$.

\begin{lemma} \label{weloc}
Let $\mathcal{M}_\mathrm{loc}$ be a Bousfield localisation of a model category $\mathcal{M}$. A fibrant object of $\mathcal{M}$ is local fibrant if and only if it is weakly equivalent in $\mathcal{M}$ to a local fibrant object.
\end{lemma}
\begin{proof}
The condition being obviously necessary, we prove its sufficiency. Suppose $X$ is a fibrant object of $\mathcal{M}$ that is weakly equivalent to a local fibrant object $Y$. Since both objects are fibrant in $\mathcal{M}$, they are connected by a span of weak equivalences $X \longleftarrow Z \longrightarrow Y$ in which the object $Z$ is fibrant. Hence it suffices to consider the two cases in which (i) there exists a weak equivalence $X \lra Y$, or (ii) there exists a weak equivalence $Y \lra X$.

In case (i), take a factorisation of  the weak equivalence $X \lra Y$ into a trivial cofibration $X \lra W$ followed by a trivial fibration $W \lra Y$. Since a trivial fibration is in particular a local fibration, $W$ is a local fibrant object. Since $X$ is fibrant, the trivial cofibration $X \lra W$ has a retraction, whence $X$ is a retract of the local fibrant object $W$, and is therefore local fibrant.

In case (ii), let $X \lra X'$ be a local fibrant replacement of $X$. The composite $Y \lra X \lra X'$ is then a local weak equivalence between local fibrant objects, and hence is a weak equivalence. It then follows from the two-of-three property that $X \lra X'$ is a weak equivalence, and so $X$ is local fibrant by case (i).
\end{proof}

One can use the following criterion involving local fibrant objects to determine when a Bousfield localisation of a cartesian model category is cartesian, at least when every object is cofibrant. Recall that a model category $\mathcal{M}$ is said to be \emph{cartesian} (or \emph{cartesian closed} \cite[Definition 7.29]{MR2342834}) if its underlying category is cartesian closed, its terminal object is cofibrant,  and the product functor $- \times - \colon \mathcal{M} \times \mathcal{M} \lra \mathcal{M}$ is a  \emph{left Quillen bifunctor} \cite[Definition 4.2.1]{MR1650134}. 

\begin{proposition} \label{cartprop}
Let $\mathcal{M}$ be a cartesian model category in which every object is cofibrant. A Bousfield localisation of $\mathcal{M}$ is cartesian if and only if the internal hom object $X^A$ is local fibrant for every object $A$ and every local fibrant object $X$ of $\mathcal{M}$.
\end{proposition}
\begin{proof}
The condition is necessary since every object is cofibrant and the internal hom functor of a cartesian model category is a right Quillen bifunctor.

To prove sufficiency, note that any Bousfield localisation of $\mathcal{M}$ inherits the properties that the terminal object is cofibrant and that the pushout-product of any two cofibrations is a cofibration; hence it remains to show that the pushout-product of a local trivial cofibration with a cofibration is a local trivial cofibration, or equivalently a local weak equivalence.

First, observe that for any local weak equivalence $f \colon A \lra B$ and any object $C$, the morphism $f\times C\colon A \times C \lra B \times C$ is a local weak equivalence. This follows from Lemma \ref{locwe}, since for any local fibrant object $X$, the function $(\Ho \mathcal{M})(B\times C,X) \lra (\Ho\mathcal{M})(A\times C,X)$ is isomorphic to the function $(\Ho\mathcal{M})(B,X^C) \lra (\Ho\mathcal{M})(A,X^C)$ by \cite[Theorem 4.3.2]{MR1650134}, and the latter function is a bijection by Lemma \ref{locwe} since $X^C$ is local fibrant by assumption. 

Now, let $f \colon A \lra B$ be a local trivial cofibration and let $g \colon C \lra D$ be a cofibration. Then in the diagram 
\begin{equation*}
\cd{
A \times C \ar[r]^-{A\times g} \ar[d]_-{f \times C} & A \times D \ar[d]^j \ar@/^1pc/[ddr]^-{f \times D} \\
B \times C \ar[r] \ar@/_1pc/[drr]_-{B \times g} & \pushoutcorner \cdot \ar[dr]|-{f \widehat{\times} g} \\
&& B \times D
}
\end{equation*}
we have that the morphisms $f\times C$, its pushout $j$, and $f \times D$ are local trivial cofibrations, and hence by the two-of-three property for local weak equivalences that the pushout-product $f \widehat{\times} g$ is a local trivial cofibration. 
\end{proof}

Let us now recall an existence theorem for Bousfield localisations due to Smith, which will enable us to recognise when a class of fibrant objects in a model category is the class of local fibrant objects for a Bousfield localisation of that model category. To state this theorem, it will be helpful to first recall some results from \cite[Chapter 5]{MR1650134} concerning the canonical enrichments of homotopy categories and derived adjunctions over the \emph{classical homotopy category}, that is, the homotopy category $\Ho\sSet_\mathrm{K}$ of the category of simplicial sets equipped with the model structure for Kan complexes, which we denote by $\mathscr{H}$. By \cite[\S IV.3]{MR0210125} (see also \cite[Theorem 4.3.2]{MR1650134}), $\mathscr{H}$ is a cartesian closed category, with terminal object $\Delta^0$, and as such may be considered as a base for enriched category theory (for which, see \cite[Chapter 1]{MR2177301}). We use underlines to indicate $\mathscr{H}$-enriched categories; in particular, we denote the self-enrichment of $\mathscr{H}$ by $\underline{\mathscr{H}}$.

By \cite[Theorem 5.5.3]{MR1650134}, for any model category $\mathcal{M}$, its homotopy category $\Ho\mathcal{M}$ admits a canonical enrichment over the cartesian closed category $\mathscr{H}$; we denote this $\mathscr{H}$-enriched category by $\underline{\Ho\mathcal{M}}$ and refer to its hom-objects as the \emph{derived hom-spaces} of the model category $\mathcal{M}$. (Note that the $\mathscr{H}$-enrichment of the homotopy category $\Ho\mathcal{M}^\mathrm{op}$ defines an $\mathscr{H}$-enriched category isomorphic to the opposite of $\underline{\Ho\mathcal{M}}$.) Furthermore, by \cite[Theorem 5.6.2]{MR1650134}, for any Quillen adjunction as on the left below,
\begin{equation*}
\xymatrix{
\mathcal{M} \ar@<-1.5ex>[rr]^-{\hdash}_-{G} && \ar@<-1.5ex>[ll]_-{F} \mathcal{N}
}
\qquad\quad\qquad
\xymatrix{
\underline{\Ho\mathcal{M}} \ar@<-1.5ex>[rr]^-{\hdash}_-{\mathbf{R}G} && \ar@<-1.5ex>[ll]_-{\mathbf{L}F} \underline{\Ho\mathcal{N}}
}
\end{equation*}
its derived adjunction underlies an $\mathscr{H}$-enriched adjunction between $\mathscr{H}$-enriched homotopy categories as on the right above. We refer to the right adjoint of this $\mathscr{H}$-enriched adjunction as the \emph{$\mathscr{H}$-enriched right derived functor} of $G$.

Now, let $S$ be a set of morphisms in a model category $\mathcal{M}$. We say that an object $X$ of $\mathcal{M}$ is \emph{$S$-local}, or \emph{local with respect to $S$}, if the induced morphism between derived hom-spaces
\begin{equation} \label{precomploc}
\underline{\Ho\mathcal{M}}(f,X) : \underline{\Ho\mathcal{M}}(B,X) \lra \underline{\Ho\mathcal{M}}(A,X)
\end{equation}
is an isomorphism in $\mathscr{H}$ for each morphism $f \colon A \lra B$ belonging to $S$ (cf.\ \cite[Definition 3.1.4]{MR1944041}). 
The following theorem gives sufficient conditions for the existence of the (necessarily unique) Bousfield localisation of $\mathcal{M}$ whose local fibrant objects are the $S$-local fibrant objects of $\mathcal{M}$; if it exists, we call this Bousfield localisation the \emph{Bousfield localisation of $\mathcal{M}$ with respect to $S$}, and denote it by $\mathrm{L}_S\mathcal{M}$ (cf.\ \cite[Definition 3.3.1]{MR1944041}).

\begin{example} \label{example2}
It follows from Proposition \ref{kanisloc} that the model structure for Kan complexes on the category of simplicial sets is the Bousfield localisation of the model structure for quasi-categories with respect to the single morphism $\Delta^1 \lra \Delta^0$.
\end{example}

\begin{remark}\label{bouslocrmk} The set of morphisms $S$ may be thought of as a ``presentation'' of the Bousfield localisation $\mathrm{L}_S\mathcal{M}$ of $\mathcal{M}$ (if it exists). In general, a Bousfield localisation $\mathcal{M}_\mathrm{loc}$ of a left proper (see below) model category $\mathcal{M}$ admits many such presentations: in particular, one can always take $S$ to be the (large) set of local weak equivalences; if the model category $\mathcal{M}_\mathrm{loc}$ is cofibrantly generated, one can take $S$ to be a small set of generating trivial cofibrations for $\mathcal{M}_\mathrm{loc}$.
\end{remark}

To state the existence theorem, we require the following technical conditions. A model category is said to be \emph{left proper} if any pushout of a weak equivalence along a cofibration is a weak equivalence (see \cite[\S13.1]{MR1944041}; any model category in which every object is cofibrant is left proper), and is said to be \emph{combinatorial} if it is cofibrantly generated and its underlying category is locally presentable (see \cite[\S2]{MR1870516}). Every model category considered in this paper is both left proper and combinatorial.

\begin{theorem}[Smith] \label{smiththm}
Let $\mathcal{M}$ be a left proper combinatorial model category and let $S$ be a small set of morphisms  in $\mathcal{M}$. Then there exists a  Bousfield localisation $\mathrm{L}_S\mathcal{M}$ of $\mathcal{M}$ whose local fibrant objects are precisely the $S$-local fibrant objects of $\mathcal{M}$. The model category $\mathrm{L}_S\mathcal{M}$  is left proper and combinatorial.
\end{theorem}
\begin{proof}
See \cite[Theorem 4.7]{MR2771591}. Note that any Bousfield localisation of a left proper model category is left proper: any local weak equivalence admits a factorisation into a local trivial cofibration followed by a weak equivalence, and hence, if $\mathcal{M}$ is left proper, so too does any pushout of a local weak equivalence along a cofibration.
\end{proof}

To determine whether an object $X$ of a model category $\mathcal{M}$ is local with respect to some set of morphisms, one needs a model for the functor 
$\underline{\Ho\mathcal{M}}(-,X) : \Ho\mathcal{M}^\mathrm{op} \lra \mathscr{H}$,
which appeared in (\ref{precomploc}). In practice, such models can be easily recognised with the help of (the dual of) the following lemma, which implies that the derived right adjoint of a Quillen adjunction $F \dashv G\colon \mathcal{M}^\mathrm{op} \lra \sSet_\mathrm{K}$ is naturally isomorphic to the functor $\underline{\Ho\mathcal{M}}(-,X) : \Ho\mathcal{M}^\mathrm{op} \lra \mathscr{H}$ if $F(\Delta^0)$ is weakly equivalent to $X$ in $\mathcal{M}$.

\begin{lemma} \label{recoghom}
Let
\begin{equation*}
\xymatrix{
\mathcal{M} \ar@<-1.5ex>[rr]^-{\hdash}_-{G} && \ar@<-1.5ex>[ll]_-{F} \sSet_{\mathrm{K}}
}
\end{equation*}
be a Quillen adjunction between a model category $\mathcal{M}$ and the category of simplicial sets equipped with the model structure for Kan complexes. For each object $A$ of $\mathcal{M}$, the following are equivalent.
\begin{enumerate}[leftmargin=*, font=\normalfont, label=(\roman*)]
\item The $\mathscr{H}$-enriched right derived functor of $G$ is $\mathscr{H}$-naturally isomorphic to the $\mathscr{H}$-enriched representable functor
$\underline{\Ho\mathcal{M}}(A,-) : \underline{\Ho \mathcal{M}} \lra \underline{\mathscr{H}}$.
\item The objects  $F(\Delta^0)$ and $A$ are weakly equivalent in the model category $\mathcal{M}$.
\end{enumerate}
\end{lemma}
\begin{proof}
By the $\mathscr{H}$-enriched derived adjunction $\mathbf{L}F \dashv \mathbf{R}G \colon \underline{\Ho\mathcal{M}} \lra \underline{\mathscr{H}}$  of the Quillen adjunction $F \dashv G$, and  since $\Delta^0$ is the terminal object of the cartesian closed category $\mathscr{H}$, there exist isomorphisms in $\mathscr{H}$
\begin{equation*}
(\mathbf{R}G)X \cong \underline{\mathscr{H}}(\Delta^0,(\mathbf{R}G)X) \cong \underline{\Ho\mathcal{M}}((\mathbf{L}F)\Delta^0,X)
\end{equation*}
$\mathscr{H}$-natural in $X \in \underline{\Ho\mathcal{M}}$.
Since $\Delta^0$ is a cofibrant object of $\sSet_\mathrm{K}$, there exists an isomorphism $(\mathbf{L}F)\Delta^0 \cong F(\Delta^0)$ in $\Ho\mathcal{M}$. Hence the $\mathscr{H}$-enriched functor $\mathbf{R}G \colon \underline{\Ho\mathcal{M}} \lra \underline{\mathscr{H}}$ is represented by the object $F(\Delta^0)$. The result then follows from the $\mathscr{H}$-enriched Yoneda lemma.
\end{proof}

We conclude this section with two criteria for detecting Quillen equivalences between Bousfield localisations, which are stated in terms of local fibrant objects. These criteria distill presumably standard arguments;  we apply them in \S\ref{secquillen} to prove the Quillen equivalences mentioned in \S\ref{secintro}. First,  we give a necessary and sufficient condition for an adjunction to remain a Quillen adjunction after Bousfield localisation. 

\begin{proposition} \label{quillenprop}
Let $F \dashv G \colon \mathcal{M} \lra \mathcal{N}$ be a Quillen adjunction between model categories, and let $\mathcal{N}_\mathrm{loc}$ be a Bousfield localisation of $\mathcal{N}$. The adjunction $F\dashv G \colon \mathcal{M} \lra \mathcal{N}_\mathrm{loc}$ is a Quillen adjunction if and only if the functor $G$ sends each fibrant object of $\mathcal{M}$ to a fibrant object of $\mathcal{N}_\mathrm{loc}$.
\end{proposition}
\begin{proof}
The condition is necessary since right Quillen functors preserve fibrant objects. To prove the converse, it suffices by \cite[Proposition 7.15]{MR2342834} to prove that $F \colon \mathcal{N}_{\mathrm{loc}} \lra \mathcal{M}$ preserves cofibrations and that $G \colon \mathcal{M} \lra \mathcal{N}_\mathrm{loc}$ preserves fibrations between fibrant objects. The first holds since $\mathcal{N}$ and $\mathcal{N}_\mathrm{loc}$ share the same class of cofibrations and since $F \colon \mathcal{N} \lra \mathcal{M}$ preserves cofibrations. The second holds since the hypothesis implies that $G$ sends each fibration between fibrant objects in $\mathcal{M}$ to a fibration between local fibrant objects in $\mathcal{N}$, which by \cite[Proposition 7.21]{MR2342834} is a fibration in $\mathcal{N}_\mathrm{loc}$.
\end{proof}

Recall that a Quillen adjunction $F \dashv G \colon \mathcal{M} \lra \mathcal{N}$ is said to be a homotopy reflection if its derived right adjoint $\mathbf{R}G \colon \Ho\mathcal{M} \lra \Ho\mathcal{N}$ is fully faithful \cite[Definition E.2.15]{joyalbarcelona}.

\begin{theorem} \label{qethm1}
Let $F \dashv G \colon \mathcal{M} \lra \mathcal{N}$ be a homotopy reflection, and let $\mathcal{N}_\mathrm{loc}$ be a Bousfield localisation of $\mathcal{N}$. The adjunction $F\dashv G \colon \mathcal{M} \lra \mathcal{N}_\mathrm{loc}$ is a Quillen equivalence if and only if the following conditions are satisfied:
\begin{enumerate}[leftmargin=*, font=\normalfont, label=(\roman*)]
\item $G$ sends each fibrant object of $\mathcal{M}$ to a fibrant object of $\mathcal{N}_\mathrm{loc}$, and 
\item for every cofibrant fibrant object $X$ of $\mathcal{N}_\mathrm{loc}$, there exists a fibrant object $A$ of $\mathcal{M}$ and a weak equivalence $X \lra GA$ in $\mathcal{N}$.
\end{enumerate}
\end{theorem}
\begin{proof}
Suppose that the adjunction $F\dashv G \colon \mathcal{M} \lra \mathcal{N}_\mathrm{loc}$ is a Quillen equivalence. Condition (i) holds by Proposition \ref{quillenprop}.  To prove condition (ii), let $X$ be a cofibrant fibrant object of $\mathcal{N}_\mathrm{loc}$. Then the composite morphism
\begin{equation*}
\cd{
X \ar[r]^-{\eta_X} & GFX \ar[r]^-{Gr} & G(FX)^f,
}
\end{equation*}
where $\eta$ is the unit of the adjunction $F \dashv G$ and $r \colon FX \lra (FX)^f$ is a fibrant replacement of $FX$ in $\mathcal{M}$, is the component of the derived unit of this Quillen equivalence at the cofibrant object $X$, and is therefore a local weak equivalence between local fibrant objects, and hence also a weak equivalence in $\mathcal{N}$.

Conversely, suppose that the conditions (i) and (ii) hold. Condition (i) implies that the adjunction $F\dashv G \colon \mathcal{M} \lra \mathcal{N}_\mathrm{loc}$ is a Quillen adjunction by Proposition \ref{quillenprop}. To show that this Quillen adjunction is a Quillen equivalence,  it suffices to show that its derived right adjoint $\mathbf{R}G \colon \Ho\mathcal{M} \lra \Ho \mathcal{N}_\mathrm{loc}$ is an equivalence of categories. Since the original homotopy reflection is equal to the composite Quillen adjunction
\begin{equation*}
\xymatrix{
\mathcal{M} \ar@<-1.5ex>[rr]^-{\hdash}_-G && \ar@<-1.5ex>[ll]_-{F} \mathcal{N_{\mathrm{loc}}} \ar@<-1.5ex>[rr]^-{\hdash}_-{1_\mathcal{N}} && \ar@<-1.5ex>[ll]_-{1_\mathcal{N}} \mathcal{N},
}
\end{equation*}
 we have that the composite of the functor $\mathbf{R}G \colon \Ho\mathcal{M} \lra \Ho \mathcal{N}_\mathrm{loc}$ with the fully faithful functor $\mathbf{R}1_\mathcal{N} \colon \Ho\mathcal{N}_\mathrm{loc} \lra \Ho\mathcal{N}$ is fully faithful, and hence that the functor $\mathbf{R}G \colon \Ho\mathcal{M} \lra \Ho \mathcal{N}_\mathrm{loc}$ is fully faithful. Since every object of $\Ho \mathcal{N}_\mathrm{loc}$ is isomorphic to a cofibrant fibrant object of $\mathcal{N}_\mathrm{loc}$, condition (ii) implies that the functor $\mathbf{R}G \colon \Ho\mathcal{M} \lra \Ho \mathcal{N}_\mathrm{loc}$ is essentially surjective on objects, and therefore an equivalence of categories.
\end{proof}

Finally, we give a necessary and sufficient condition for a Quillen equivalence to remain a Quillen equivalence after Bousfield localisation. 

\begin{theorem} \label{qethm2}
Let $F \dashv G \colon \mathcal{M} \lra \mathcal{N}$ be a Quillen equivalence between model categories, and let $\mathcal{M}_\mathrm{loc}$ and $\mathcal{N}_\mathrm{loc}$ be Bousfield localisations of $\mathcal{M}$ and $\mathcal{N}$ respectively. The adjunction $F \dashv G \colon \mathcal{M}_\mathrm{loc} \lra \mathcal{N}_\mathrm{loc}$ is a Quillen equivalence if and only if a fibrant object $A$ of $\mathcal{M}$ is fibrant in $\mathcal{M}_\mathrm{loc}$ precisely when $GA$ is fibrant in $\mathcal{N}_\mathrm{loc}$.
\end{theorem}
\begin{proof}
The composite Quillen adjunction
\begin{equation} \label{compadj}
\xymatrix{
\mathcal{M}_\mathrm{loc} \ar@<-1.5ex>[rr]^-{\hdash}_-{1_{\mathcal{M}}} && \ar@<-1.5ex>[ll]_-{1_{\mathcal{M}}} \mathcal{M} \ar@<-1.5ex>[rr]^-{\hdash}_-{G} && \ar@<-1.5ex>[ll]_-{F} \mathcal{N}
}
\end{equation}
is a homotopy reflection, since it is a composite of homotopy reflections. Hence the adjunction $F \dashv G \colon \mathcal{M}_\mathrm{loc} \lra \mathcal{N}_\mathrm{loc}$ is a Quillen equivalence if and only if the homotopy reflection (\ref{compadj}) satisfies the conditions of Theorem \ref{qethm1}. It will therefore suffice to show that these conditions are equivalent to that of the present theorem.

Suppose that the homotopy reflection (\ref{compadj}) satisfies the conditions (i) and (ii) of Theorem \ref{qethm1}, and let $A$ be a fibrant object of $\mathcal{M}$. If $A$ is fibrant in $\mathcal{M}_\mathrm{loc}$, then by condition (i), $GA$ is a fibrant object of $\mathcal{N}_\mathrm{loc}$. Conversely, suppose $GA$ is a fibrant object of $\mathcal{N}_\mathrm{loc}$. Let $(GA)^c \lra GA$ be a cofibrant replacement of $GA$ in $\mathcal{N}$, chosen to be a trivial fibration. Then $(GA)^c$ is a cofibrant fibrant object of $\mathcal{N}_\mathrm{loc}$, and so by condition (ii), there exists a fibrant object $B$ of $\mathcal{M}_\mathrm{loc}$ and a weak equivalence $(GA)^c \lra GB$ in $\mathcal{N}$. Hence the objects $GA$ and $GB$ are weakly equivalent in $\mathcal{N}$, and since the derived right adjoint of the Quillen equivalence  $F \dashv G\colon \mathcal{M} \lra \mathcal{N}$ is fully faithful, it follows that $A$ and $B$ are weakly equivalent in $\mathcal{M}$. Hence $A$ is a fibrant object in $\mathcal{M}_\mathrm{loc}$ by Lemma \ref{weloc}.

To prove the converse, suppose that a fibrant object $A$ of $\mathcal{M}$ is fibrant in $\mathcal{M}_\mathrm{loc}$ precisely when $GA$ is fibrant in $\mathcal{N}_\mathrm{loc}$. One half of this assumption is precisely condition (i) of Theorem \ref{qethm1}. To verify condition (ii) of that theorem, let $X$ be a cofibrant fibrant object of $\mathcal{N}_\mathrm{loc}$. For any fibrant replacement $r \colon FX \lra (FX)^f$ of $FX$ in $\mathcal{M}$, the composite morphism 
\begin{equation*}
\cd{
X \ar[r]^-{\eta_X} & GFX \ar[r]^-{Gr} & G(FX)^f
}
\end{equation*}
gives the component at $X$ of the derived unit of the original Quillen equivalence, and is thus a weak equivalence in $\mathcal{N}$. Hence $G(FX)^f$ is weakly equivalent in $\mathcal{N}$ to the fibrant object $X$ of $\mathcal{N}_\mathrm{loc}$, and is therefore itself fibrant in $\mathcal{N}_\mathrm{loc}$ by Lemma \ref{weloc}. The other half of our assumption now implies that $(FX)^f$ is a fibrant object of $\mathcal{M}_\mathrm{loc}$, thus verifying condition (ii).
\end{proof}

\bibliographystyle{alpha}

\end{document}